\documentclass[12pt,reqno]{amsart}

\usepackage{amsmath,amssymb,mathrsfs}
\usepackage{graphicx,cite,times}
\usepackage{hyperref}
\hypersetup{
    colorlinks=true,
    linkcolor=blue,
    filecolor=gray,
    urlcolor=blue,
    citecolor=blue,
}

\usepackage[doipre={doi:~}]{uri}
\usepackage{cases}
 \usepackage{dsfont}
 \usepackage{appendix}
 \usepackage{color}
\newtheorem{theo}{Theorem}[section]

\newtheorem{lemm}[theo]{Lemma}
\newtheorem{prop}[theo]{Proposition}
\newtheorem{rema}[theo]{Remark}
\newtheorem{defi}[theo]{Definition}

\newtheorem{assu}{Assumption}

\numberwithin{equation}{section}

\begin{document}

\title[]{
Symmetrization of the Maxwell--Neumann--Poincar\'e operator, spectral decomposition in $\mathbf{H}(\mathrm{curl},D)$ traces, and boundary localisation of SPR{\small S}}

\author{Bochao Chen}
\address{School of
Mathematics and Statistics, Center for Mathematics and
Interdisciplinary Sciences, Northeast Normal University, Changchun, Jilin 130024, P.R.China. }
\email{chenbc758@nenu.edu.cn}

\author{Yixian Gao}
\address{School of
Mathematics and Statistics, Center for Mathematics and
Interdisciplinary Sciences, Northeast Normal University, Changchun, Jilin 130024, P.R.China. }
\email{gaoyx643@nenu.edu.cn}

\author{Hongyu Liu}
\address{Department of Mathematics, City University of Hong Kong, Kowloon, Hong Kong, China }
\email{hongyu.liuip@gmail.com, hongyliu@cityu.edu.hk}

\thanks{The research of BC was supported by the National
NSFC (project number, 12471181) and the Science and Technology Development Plan Project of Jilin Province, China (project number, 20260101045JJ). The research of YG was supported by the NSFC
(project number, 12371187).
The research of HL was supported by Hong Kong RGC General Research Funds (project numbers, 11311122,
11304224, and 11300821) and the NSFC-RGC Joint Research Grant (project number,
 N\_CityU101/21), and the ANR/RGC Joint Research Grant (project number, A\_CityU203/19).}

\subjclass[2020]{47A75, 47G10}

\keywords{Maxwell--Neumann--Poincar\'{e} operators, Symmetrization, Boundary localization, Weak SPRs}

\begin{abstract}

The Neumann--Poincar\'{e} (NP) operator, a fundamental operator in potential theory, has attracted renewed attention for its central role in the analysis of surface plasmon resonances (SPRs).  SPRs, characterized by non-radiative electromagnetic waves at material interfaces with opposing permittivities, underpin advanced technologies such as bio-sensing and cloaking devices. While spectral properties of the scalar NP operator and SPR dynamics for scalar waves are well-established, their vectorial counterparts in Maxwell's framework remain poorly understood. The present work bridges this gap by introducing a symmetrization principle for the matrix-valued Maxwell--Neumann--Poincar\'{e}  (MNP) operator, enabling a spectral decomposition of traces in the $\mathbf{H}(\mathrm{curl},D)$ space, which is a foundational advance for electromagnetic theory. Building on this framework, we rigorously characterize the quantum-ergodic localization of  the weak plasmon sequence at material boundaries in the full Maxwell system, thereby settling a long-standing question concerning their quantitative description.

\end{abstract}

\maketitle

\section{introduction}
\subsection{Statement of main results and related discussion}
This article is devoted to two primary objectives. The first concerns the establishment of the spectral properties of the MNP operator and its adjoint. The second aims to employ the spectral decomposition of this operator to prove a boundary localization theorem for SPRs in the setting of quantum ergodicity. This result applies to generic nanoparticles when excited by appropriate incident fields.

Let $D$ be a bounded simply connected domain in $\mathbb{R}^3$ whose boundary $\partial D$ is of class  $C^{2}$. Denote by $\boldsymbol{\nu}$ be the exterior unit normal vector to the surface of the object $D$. The fundamental solution of the Laplacian in \(\mathbb{R}^3\) is given by
\begin{align}\label{G:sca-Fun}
G(\boldsymbol{x},\boldsymbol{y})=-
\frac{1}{4\pi|\boldsymbol{x}-\boldsymbol{y}|}, \quad \boldsymbol{x},\boldsymbol{y}\in\mathbb{R}^3~\text{with}~\boldsymbol{x}\neq \boldsymbol{y}.
\end{align}
The  static MNP operator   $\mathcal{M}_{\partial D}$   on $\partial D$ is defined by 
\begin{align}\label{OP:M}
\mathcal{M}_{\partial D}:\mathbf{H}^{-\frac{1}{2}}_t(\mathrm{div},\partial D)&  \rightarrow \mathbf{H}^{-\frac{1}{2}}_t(\mathrm{div},\partial D)\nonumber\\
\boldsymbol{\phi}&\mapsto\mathcal{M}_{\partial D}[\boldsymbol{\phi}](\boldsymbol{x})
=\int_{\partial D}\boldsymbol{\nu}_{\boldsymbol{x}}\times \nabla_{\boldsymbol x}\times  G(\boldsymbol{x},\boldsymbol{y})\boldsymbol{\phi}(\boldsymbol{y})\mathrm{d}s(\boldsymbol{y}),
\end{align}
where  the  trace space $\mathbf{H}^{-\frac{1}{2}}_t(\mathrm{div},\partial D)$ is precisely 
defined in  Section \ref{sec:3}. In the following, the definitions of the operators $\mathcal{N}_{\partial D}$, $\mathcal{N}^{-1}_{\partial D}$, $\mathcal{Q}_{\partial D}$, $\mathcal{Q}^{-1}_{\partial D}$, $\mathcal{K}^*_{\partial D} $, and $ \mathcal{S}_{\partial D}$, along with the spaces $\mathbf{H}^{-\frac{1}{2}}_t(\mathrm{curl},\partial D)$,  $\overset{\longrightarrow}{\mathrm{curl}}_{\partial D}(H^{\frac32}_{\star}(\partial D))$, $\overset{\longrightarrow}{\mathrm{curl}}_{\partial D}(H^{\frac{1}{2}}(\partial D))$, $\nabla_{\partial D}(H^{\frac32}_{\star}(\partial D))$, and
$\nabla_{\partial D}(H^{\frac12}(\partial D))$, are rigorously defined in Section \ref{sec:3}. 

With respect to the bilinear form $\langle \cdot,\cdot\rangle_{\mathbf{H}^{-\frac{1}{2}}_t(\mathrm{div},\partial D),\mathbf{H}^{-\frac{1}{2}}_t(\mathrm{curl},\partial D)}$ defined in \eqref{eq:bilinear}, the adjoint of $\mathcal{M}_{\partial D}$, denoted by $\mathcal{M}^*_{\partial D}$, is given by
\begin{align}\label{eq:adjoint}
\mathcal{M}^*_{\partial D}:\mathbf{H}^{-\frac{1}{2}}_t(\mathrm{curl},\partial D)&\rightarrow \mathbf{H}^{-\frac{1}{2}}_t(\mathrm{curl},\partial D)\nonumber\\
\boldsymbol \phi&\mapsto\boldsymbol{\nu}\times\mathcal{M}_{\partial D}[\boldsymbol{\nu}\times\boldsymbol\phi].
\end{align}

We first establish a symmetrization principle formalized in the following theorem. 

\begin{theo}\label{th:symmetrisation1}
The MNP operator $\mathcal{M}_{\partial D}$ expressed as \eqref{OP:M} satisfies the following properties:
\begin{enumerate}
\item[\rm{(1)}] (Calder\'{o}n Identity) The  commutation relation
\begin{align}
\mathcal{N}_{\partial D}\mathcal{M}^*_{\partial D}&=\mathcal{M}_{\partial D}\mathcal{N}_{\partial D}\label{KK:Cald3}
\end{align}
holds on $\overset{\longrightarrow}{\mathrm{curl}}_{\partial D}(H^{\frac32}_{\star}(\partial D))$.

\item[\rm{(2)}] (Self-adjointness) The operator 
\[\mathcal{M}_{\partial D}:\overset{\longrightarrow}{\mathrm{curl}}_{\partial D}(H^{\frac{1}{2}}(\partial D))\rightarrow\overset{\longrightarrow}{\mathrm{curl}}_{\partial D}(H^{\frac{1}{2}}(\partial D))\] 
is self-adjoint with respect  to the inner product $\langle\cdot,\cdot\rangle_{\mathcal{N}^{-1}_{\partial D};\overset{\longrightarrow}{\mathrm{curl}}_{\partial D}(H^{\frac{1}{2}}(\partial D))}$ defined in \eqref{E:inner}.
\end{enumerate}
\end{theo}

Remark that the Calder\'{o}n-type identity \eqref{KK:Cald3} we employ were established in \cite{Ammari2016Surface}. However, the bilinear form \eqref{E:inner} we use differs from that in the reference. This necessitates additional considerations regarding the invertibility of the relation operator $\mathcal{N}_{\partial D}$, as well as the equivalence between the inner product induced by this operator and the original norm on the underlying space. Building on these foundations, we aim to further investigate the quantum-ergodic localization of weak SPRs, a topic that has attracted considerable attention in recent years.

\begin{theo}\label{th:symmetrisation2}
The adjoint operator $\mathcal{M}^*_{\partial D}$, defined  in \eqref{eq:adjoint}, satisfies the following properties:
\begin{enumerate}
\item[\rm{(1)}]  (Calder\'{o}n Identity) On the space $\nabla_{\partial D}(H^{\frac32}_{\star}(\partial D))$, we have
\begin{align}
\mathcal{M}^*_{\partial D}\mathcal{Q}_{\partial D}&=\mathcal{Q}_{\partial D}\mathcal{M}_{\partial D}.\label{KK:Cald4}
\end{align}

\item[\rm{(2)}]  (Self-adjointness) Restricted to $\nabla_{\partial D}(H^{\frac{1}{2}}(\partial D))$, 
 the operator 
 \[\mathcal{M}^*_{\partial D}:\nabla_{\partial D}(H^{\frac{1}{2}}(\partial D))\rightarrow\nabla_{\partial D}(H^{\frac{1}{2}}(\partial D))\] 
 is self-adjoint with respect to the inner product  $\langle\cdot,\cdot\rangle_{\mathcal{Q}^{-1}_{\partial D};\nabla_{\partial D}(H^{\frac{1}{2}}(\partial D))}$ introduced  in  \eqref{EE:inner}.
\end{enumerate}
\end{theo}

Let $\{\theta_j\}_{j\in\mathbb{Z}_+}$ represent the eigenfunctions of the NP operator 
$\mathcal{K}^*_{\partial D}$, normalized with respect to the inner product
\begin{align}\label{inner:S}
\langle u,v\rangle_{\mathcal{S}_{\partial D};H^{-\frac{1}{2}}(\partial D)}:=-\int_{\partial D}u(\boldsymbol x) \mathcal{S}_{\partial D}[v](\boldsymbol x) \mathrm{d} s(\boldsymbol x),\quad \forall u,v\in H^{-\frac{1}{2}}(\partial D).
\end{align} 
Denote by $\mathcal{P}_{\boldsymbol{\mathcal{H}}}$ the orthogonal projection onto the space $\boldsymbol{\mathcal{H}}$.  We now present the main theorem on spectral decompositions.

\begin{theo}\label{th:decomposition1}
The MNP operator \(\mathcal{M}_{\partial D}\) on the space  
$\overset{\longrightarrow}{\mathrm{curl}}_{\partial D}(H^{\frac12}(\partial D))$ 
satisfies the following properties:

\begin{enumerate}
\item[\rm{(1)}]  
The operator $\mathcal{M}_{\partial D}$ admits a sequence of eigenvalues $\left\{\lambda_{j}\right\}_{j \in \mathbb{Z}_+}$ with
\begin{align*}
\lambda_j\rightarrow 0\quad\text{as}\quad j \rightarrow+\infty.
\end{align*}

\item[\rm{(2)}]  Let $\left\{\boldsymbol{\varphi}_j\right\}_{j \in \mathbb{Z}_+}$ be the eigenfunctions of $\mathcal{M}_{\partial D}$, normalized with respect to the inner product $\langle\cdot,\cdot\rangle_{\mathcal{N}^{-1}_{\partial D};\overset{\longrightarrow}{\mathrm{curl}}_{\partial D}(H^{\frac{1}{2}}(\partial D))}$.
Then
\begin{align}\label{E:eigenfunction}
\boldsymbol{\varphi}_j=\overset{\longrightarrow}{\mathrm{curl}}_{\partial D}\mathcal{S}_{\partial D}[\theta_j],\quad \forall j\in \mathbb{Z}_+.
\end{align}

\item[\rm{(3)}] The functions $\left\{\mathcal{N}^{-1}_{\partial D}[\boldsymbol{\varphi}_j]\right\}_{j\in \mathbb{Z}_+}$ are normalized eigenfunctions of the composite operator
\begin{align*}
\mathcal{P}_{\overset{\longrightarrow}{\mathrm{curl}}_{\partial D}(H^{\frac32}_{\star}(\partial D))}\circ\mathcal{M}^*_{\partial D}\Big|_{\overset{\longrightarrow}{\mathrm{curl}}_{\partial D}(H^{\frac32}_{\star}(\partial D))}:\overset{\longrightarrow}{\mathrm{curl}}_{\partial D}(H^{\frac{3}{2}}_{\star}(\partial D)) \rightarrow \overset{\longrightarrow}{\mathrm{curl}}_{\partial D}(H^{\frac{3}{2}}_{\star}(\partial D)),
\end{align*}
where the normalization is taken with respect to the inner product $\langle\cdot,\cdot\rangle_{\mathcal{N}_{\partial D};\overset{\longrightarrow}{\mathrm{curl}}_{\partial D}(H^{\frac32}_{\star}(\partial D))}$ defined in \eqref{E:Product}. 

\end{enumerate}
\end{theo}

\begin{theo}\label{th:decomposition2}
The adjoint operator $\mathcal{M}^*_{\partial D}$ on the space  
$\nabla_{\partial D}(H^{\frac{1}{2}}(\partial D))$ satisfies the following spectral properties:

\begin{enumerate}
\item[\rm{(1)}]   The operator $\mathcal{M}^*_{\partial D}$ admits a sequence of eigenvalues $\left\{\lambda^*_{j}\right\}_{j\in \mathbb{Z}_+}$ with 
\begin{align*}
\lambda^*_j\rightarrow 0 \quad \text{as}\quad j\rightarrow+\infty.
\end{align*}

\item[\rm{(2)}] Let $\left\{\boldsymbol{\psi}_j\right\}_{j \in \mathbb{Z}_+}$ be the normalized eigenfunctions of $\mathcal{M}^*_{\partial D}$ with respect to the inner product $\langle\cdot,\cdot\rangle_{\mathcal{Q}^{-1}_{\partial D};\nabla_{\partial D}(H^{\frac{1}{2}}(\partial D))}$. They satisfy
\begin{align*}
\boldsymbol{\psi}_j = \nabla_{\partial D}\mathcal{S}_{\partial D}[\theta_j], \quad \forall j \in \mathbb{Z}_+.
\end{align*}

\item[\rm{(3)}] The functions $\left\{\mathcal{Q}^{-1}_{\partial D}[\boldsymbol{\psi}_j]\right\}_{j \in \mathbb{Z}_+}$ are normalized eigenfunctions, in the sense of the inner product $\langle\cdot,\cdot\rangle_{\mathcal{Q}_{\partial D};\nabla_{\partial D}(H^{\frac32}_{\star}(\partial D))}$ (see \eqref{EE:product}), of the composite operator
\begin{align*}
\mathcal{P}_{\nabla_{\partial D}(H^{\frac32}_{\star}(\partial D))}\circ\mathcal{M}_{\partial D}\Big|_{\nabla_{\partial D}(H^{\frac32}_{\star}(\partial D))}:\nabla_{\partial D}(H^{\frac32}_{\star}(\partial D))\rightarrow\nabla_{\partial D}(H^{\frac{3}{2}}_{\star}(\partial D)).
\end{align*}
\end{enumerate} 
\end{theo}

By exploiting the spectral decomposition of the MNP operator, when weak SPRs (Definition~\ref{def:weak}) occur, we establish the boundary localization theorem \ref{th:location1} for the weak plasmon sequence, as seen in \eqref{E:mode} and \eqref{H:mode}, in the setting of quantum ergodicity. The result holds for general nanoparticles illuminated by suitable incident fields. In the particular case of spherical nanoparticles, we further demonstrate that the weak SPRs exhibit strict surface concentration via spherical harmonic analysis.  We begin by recalling the notion of almost-sure convergence for sequences in the context of quantum ergodicity; see \cite{Sunada1997quantum} for details.

\begin{defi}\label{def:almost_sure_o}
Let $\left\{c_j\right\}_{j \in \mathbb{Z}_{+}}$ be a sequence of real (or complex) numbers and let $\kappa \geq 0$. 
We say that $c_j = o(j^{-\kappa})$ almost surely as $j \to +\infty$ if
\begin{align*}
\lim_{\sigma\rightarrow0^+} \limsup_{N \rightarrow +\infty} 
\frac{\#\{j \leq N : |c_j| > \sigma j^{-\kappa}\}}{N} = 0.
\end{align*}

\end{defi}

\begin{theo}\label{th:location1}
Let $\omega > 0$ be fixed, and let $B \subset \mathbb{R}^3$ be a bounded simply connected domain containing the origin and having a $C^{2}$ boundary. For a center $\boldsymbol{z} \in \mathbb{R}^3$ and a sufficiently small $\delta > 0$, let the nanoparticle $D$ be given by 
\begin{align}\label{eq:D}
D:=\boldsymbol{z}+\delta B.
\end{align}
Under Assumption \ref{assumption2}, if weak SPRs occur for the Maxwell system \eqref{eq:maxwell}, then the associated weak plasmon sequence 
$\{(\boldsymbol{E}_j,\boldsymbol{H}_j)\}_{j\in\mathbb{Z}_+}$, defined in \eqref{E:mode} and \eqref{H:mode}, satisfies, for any 
\(\epsilon > 0\),
\begin{align*}
\|\boldsymbol{E}_{j}\|_{L^2_{\mathrm{loc}}(\mathbb{R}^3 \setminus \mathcal{T}_\epsilon(\partial D))^3}
= o(j^{-1/2}), \quad
\|\boldsymbol{H}_{j}\|_{L^2_{\mathrm{loc}}(\mathbb{R}^3 \setminus \mathcal{T}_\epsilon(\partial D))^3}
= o(j^{-1/2})
\end{align*}
almost surely as $j \rightarrow +\infty$ in the sense of Definition~\ref{def:almost_sure_o}, where 
$\mathcal{T}_\epsilon(\partial D)$ denotes the $\epsilon$-tubular neighborhood 
of $\partial D$:
\begin{align*}
\mathcal{T}_\epsilon(\partial D):= 
\left\{\boldsymbol{x} \in \mathbb{R}^3:\operatorname{dist}(\boldsymbol{x},\partial D)<\epsilon\right\}.
\end{align*}

\end{theo}
The precise formulation of Assumption \ref{assumption2} is
presented in Section~\ref{sec:5}.

\begin{theo}\label{th:location2}
Fix $\omega > 0$. Let $B \subset \mathbb{R}^3$ be the unit ball centered at the origin. For a sufficiently small $\delta > 0$, the corresponding spherical nanoparticle is defined as $D:=\delta B$. 
Under Assumption \ref{assumption3}, if the 
Maxwell system \eqref{eq:maxwell} exhibits weak SPRs, then the 
weak plasmon sequence $\left\{(\boldsymbol{E}_{l,n}, \boldsymbol{H}_{l,n})\right\}_{l=1,2;n \in \mathbb{Z}_+}$, defined in \eqref{eq:E} and \eqref{eq:H}, 
satisfies, for $l=1,2$,
\begin{align*}
\boldsymbol{E}_{l,n}(\boldsymbol{x})\rightarrow 0, \quad 
\boldsymbol{H}_{l,n}(\boldsymbol{x})\rightarrow 0 \quad \text{as}\quad n \rightarrow +\infty,
\end{align*}
with exponentially fast convergence, uniformly for $\boldsymbol{x}$ in compact subsets 
of $\mathbb{R}^3 \setminus \overline{D}$ or $\boldsymbol{x} \in D$.
\end{theo}

The precise formulation of Assumption~\ref{assumption3} is given in Section~\ref{sec:6}.

\subsection{Research background and state of the art} 

The analysis of SPRs is predominantly carried out using layer potential techniques. This approach involves deriving boundary integral equations under resonance conditions, thereby recasting the problem as a transmission eigenvalue problem on the boundary. Key steps include establishing the self-adjointness of the NP operator via Calder\'{o}n-type symmetrization, proving its compactness in suitable Hilbert spaces (which ensures eigenvalues converge to zero), and expressing plasmons in terms of the corresponding eigenfunctions.

The spectral analysis of NP operators continues to attract significant attention, particularly due to their applications in plasmon resonance phenomena and cloaking based on anomalous localized resonance \cite{Ammari2013spectral,Ando2016analysis}. Considerable research has been devoted to NP operators associated with the Laplace equation \cite{Bonnetier2013spectrum,Perfekt2014spectral,Ando2022spectral,Ando2019spectral,ji2023spectral}; their spectral properties are fundamental to advances in nanophotonics and metamaterials. More recently, attention has turned to elastic NP operators, whose non-compact nature poses distinct analytical challenges \cite{Ando2019elastic,Ando2020convergence,Fukushima2024decomposition}. In this context, spherical domains provide tractable cases for explicit spectral decomposition \cite{deng2019on,deng2020analysis}. For Maxwell equations, however, symmetrizing even the static MNP operator remains a significant challenge. Early efforts circumvented this issue by using layer‑potential theory \cite{Ammari2016Surface,Ammari2016Mathematical}. In the latter work, the authors also derived a set of non-orthogonal eigenfunctions for the associated equivalent operator. Resonance frequencies for Maxwell's equations were later computed numerically in \cite{Costabel2003computation}. Recently, an analysis of electromagnetic scattering resonances and Fano-type resonances in nanoparticles with high refractive indices was provided by \cite{Ammari2023analysis,cao2023electromagnetic,Ammari2024fano}.

In this paper, we focus on the boundary localization of weak SPRs in the quantum ergodic sense. This phenomenon has received considerable attention in recent years. For background on related results in the Laplace equation setting, we refer to \cite{Ando2021surface}. Anderson localization for the Helmholtz equation has been investigated in \cite{Ammari2024anderson}. However, the study of the boundary localization in the Maxwell system remains challenging. To address this class of problems, we draw inspiration from the approach developed in \cite{chow2023surface} for the Helmholtz equation. In that work, the boundary localization of transmission eigenfunctions was analyzed under weak resonance conditions. Numerical findings regarding this phenomenon can be found in \cite{chow2021surface}. Building on these ideas, we formulate and study the corresponding weak-resonance problem within the full Maxwell framework, with the goal of establishing analogous results for weak SPRs. 

Our analysis, which leverages the orthogonality of the eigenfunctions of the MNP operator, addresses three core challenges:
\begin{itemize}
    \item \textbf{Symmetrization}: Overcoming the non-self-adjoint character  of the MNP operator by constructing tailored bilinear forms and Calder\'on-type identities in the trace space.
    \item \textbf{Spectral Analysis}: Determining the  eigenvalue distribution through a novel Helmholtz decomposition that links the MNP operator to its scalar NP counterpart, while studying the compactness of its restriction to a subspace by using the equivalence of inner-product-induced norms.
    \item \textbf{Boundary Localization}: Developing spectral decomposition methods to rigorously establish the boundary concentration of plasmonic modes, unifying abstract operator theory with layer potential techniques for both general smooth domains and explicit spherical cases.
\end{itemize}
Although the Calder\'{o}n identity employed in this work was originally established in \cite{Ammari2016Surface}, the bilinear form adopted here differs from that in the aforementioned reference. Moreover, we are required to establish the invertibility of the operator that induces the inner product, along with the equivalence between the induced inner product and the original norm. We then proceed to investigate the quantum-ergodic localization of weak SPRs.

The organization of this paper is as follows. Section \ref{sec:3} introduces the fundamental notation and mathematical preliminaries. Section \ref{sec:4} is devoted to proving  Theorems \ref{th:symmetrisation1}--\ref{th:decomposition2}. Specifically, we establish the symmetrization and spectral decomposition theory for the static MNP operator along with its adjoint on general domains. Building on these results, Section \ref{sec:5} demonstrates quantum ergodic boundary localization in electromagnetic scattering problems for general geometries, thereby completing the proof of Theorem \ref{th:location1}. Finally, Section \ref{sec:6} specializes to spherical geometries, where precise exponential decay estimates for weak SPRs are derived, completing the proof of Theorem~\ref{th:location2}.

\section{Preliminary knowledge}\label{sec:3}

This section summarizes the basic notation and key boundary integral operators used throughout this work. Their fundamental properties, which will serve as the foundation for subsequent analysis, are also presented.

\subsection{Layer Potential Formulation}

 We begin by introducing the necessary notation for the key function spaces involved. 
 Let $ H^{s}(\partial D)$ be the standard Sobolev space of order $s\in\mathbb{R}$ on $\partial D$. Obviously, the spaces $ H^{s}(\partial D)$ and $ H^{-s}(\partial D)$ are dual 
 with respect to the pivot space  $L^2(\partial D)$.  Denote by $H^{s}_{\star}(\partial D)$ the subspace  of $H^{s}(\partial D)$ consisting of functions with zero mean:
\begin{align*}
H^{s}_{\star}(\partial D)&:=\left\{u\in H^{s}(\partial D):\int_{\partial D}u(\boldsymbol x) \mathrm{d}s(\boldsymbol x)=0\right\}.
\end{align*}
Introduce the Hilbert space $\mathbf{H}^s_t(\partial D)$ of tangential vector fields in $H^s(\partial D)^3$, denoted by
\begin{align*}
\mathbf{H}^s_t(\partial D):=\left\{\boldsymbol{\phi}\in H^s(\partial D)^3: \boldsymbol{\nu}\cdot \boldsymbol{ \phi}=0\right\}.
\end{align*}
Moreover, we define the local Sobolev space $H^1_{\mathrm{loc}}(\mathbb{R}^3)^3$ as the space of functions whose restriction to any bounded subset of $\mathbb{R}^3$ belongs to $H^1(\mathbb{R}^3)^3$. 

Let $\nabla_{\partial D}, \nabla_{\partial D}\cdot$ and $\Delta_{\partial D}$  denote the surface gradient, surface divergence and Laplace--Beltrami operator on the surface $\partial D$, respectively.  Following \cite{Ammari2016Mathematical,Nedelec2002acoustic},
we define the vectorial and scalar surface curl operators by
\begin{align*}
\begin{aligned}
\overset{\longrightarrow}{\mathrm{curl}}_{\partial D}\phi&:=- \boldsymbol{\nu}\times \nabla_{\partial D}\phi, &&\forall\phi\in H^{\frac{1}{2}}(\partial D),\\
\mathrm{curl}_{\partial D}\boldsymbol{\phi}&:=\boldsymbol{\nu}\cdot(\nabla_{\partial D}\times\boldsymbol{\phi}),&&\forall \boldsymbol{\phi}\in \mathbf{H}^{-\frac{1}{2}}_t(\partial D).
\end{aligned}
\end{align*}
Furthermore, it can be obtained from \cite[p.237]{Colton1998Inverse} that for $\tilde{\boldsymbol{g}}\in C^1(\mathbb{R}^3\backslash\overline{D})\cap C^0(\mathbb{R}^3\backslash D)$,
\begin{align}\label{eq:DC}
\nabla_{\partial D}\cdot(\boldsymbol{\nu}\times \tilde{\boldsymbol{g}})=-\boldsymbol\nu\cdot (\nabla\times \tilde{\boldsymbol{g}}).
\end{align}
Consequently, the following identities hold:
\begin{align*}
\begin{aligned}
\nabla_{\partial D}\cdot\nabla_{\partial D}&=\Delta_{\partial D},&&\mathrm{curl}_{\partial D}\overset{\longrightarrow}{\mathrm{curl}}_{\partial D}=-\Delta_{\partial D},\\
\mathrm{curl}_{\partial D}\nabla_{\partial D}&=0,&&\nabla_{\partial D}\cdot\overset{\longrightarrow}{\mathrm{curl}}_{\partial D}=0.
\end{aligned}
\end{align*}

We also define the following trace spaces:
\begin{align*}
\mathbf{H}^{-\frac{1}{2}}_t(\mathrm{div},\partial D)&:=\left\{\boldsymbol{f}\in \mathbf{H}^{-\frac{1}{2}}_t(\partial D):\nabla_{\partial D}\cdot\boldsymbol{f}\in H^{-\frac{1}{2}}(\partial D)\right\},\\
\mathbf{H}^{-\frac{1}{2}}_t(\mathrm{curl},\partial D)&:=\left\{\boldsymbol{g}\in \mathbf{H}^{-\frac{1}{2}}_t(\partial D):\mathrm{curl}_{\partial D}\boldsymbol{g}\in H^{-\frac{1}{2}}(\partial D)\right\},
\end{align*}
endowed with the corresponding norms:
\begin{align}
\|\boldsymbol{f}\|_{\mathbf{H}^{-\frac{1}{2}}_t(\mathrm{div},\partial D)}&:=\left(\|\boldsymbol{f}\|^2_{H^{-\frac{1}{2}}(\partial D)^3}+\|\nabla_{\partial D}\cdot \boldsymbol{f}\|^2_{H^{-\frac{1}{2}}(\partial D)}\right)^{\frac{1}{2}},\label{norm:div}\\
\|\boldsymbol{g}\|_{\mathbf{H}^{-\frac{1}{2}}_t(\mathrm{curl},\partial D)}&:=\left(\|\boldsymbol{g}\|^2_{H^{-\frac{1}{2}}(\partial D)^3}+\|\mathrm{curl}_{\partial D}\boldsymbol{g}\|^2_{H^{-\frac{1}{2}}(\partial D)}\right)^{\frac{1}{2}}.\label{norm:curl}
\end{align}
\begin{rema}
The dual space of $\mathbf{H}^{-\frac{1}{2}}_t(\mathrm{div},\partial D)$ is $\mathbf{H}^{-\frac{1}{2}}_t(\mathrm{curl},\partial D)$ with the duality pairing
\begin{align}\label{eq:bilinear}
\langle \boldsymbol{f},\boldsymbol g\rangle_{\mathbf{H}^{-\frac{1}{2}}_t(\mathrm{div},\partial D),\mathbf{H}^{-\frac{1}{2}}_t(\mathrm{curl},\partial D)}:=\int_{\partial D}\boldsymbol f(\boldsymbol{x})\cdot \boldsymbol g(\boldsymbol{x}) \mathrm{d}s(\boldsymbol{x})
\end{align}
given by the $L^2$ bilinear form on $\mathbf{L}^2_t(\partial D)$.  This duality result is given by  \cite[p.241]{Colton1998Inverse}.
\end{rema}
Let $\mathbf{H}(\mathrm{curl},D)$ be the space defined by
\begin{align*}
\mathbf{H}(\mathrm{curl},D):=\left\{\boldsymbol u\in L^2(D)^3:\nabla\times \boldsymbol{u}\in L^2(D)^3\right\},
\end{align*}
endowed with the norm
\begin{align*}
\|\boldsymbol{u}\|_{\mathbf{H}(\mathrm{curl},D)}:=\left(\|\boldsymbol{u}\|^2_{L^2(D)^3}+\|\nabla\times\boldsymbol{u}\|^2_{L^2(D)^3}\right)^{\frac{1}{2}}.
\end{align*}
Similarly, denote by $\mathbf{H}_{\mathrm{loc}}(\mathrm{curl},\mathbb{R}^3)$ the space of functions locally in $\mathbf{H}(\mathrm{curl},\mathbb{R}^3)$.

We now introduce the boundary integral operators. For a scalar density $\phi\in H^{-\frac{1}{2}}(\partial D)$, the single-layer potential operator $\mathcal{S}_{\partial D}$ and the NP operator $\mathcal{K}^{*}_{\partial D}$ are defined by
\begin{align}
\mathcal{S}_{\partial D}: H^{-\frac{1}{2}}(\partial D)&\rightarrow H^{\frac{1}{2}}(\partial D)\text{ or } H^1_{\mathrm{loc}}(\mathbb{R}^3)\nonumber\\
\phi&\mapsto\mathcal{S}_{\partial D}[\phi](\boldsymbol{x})=\int_{\partial D}G(\boldsymbol{x},\boldsymbol{y})\phi(\boldsymbol{y})\mathrm{d}s(\boldsymbol{y});\label{ssp}\\
\mathcal{K}^{*}_{\partial D}:H^{-\frac{1}{2}}(\partial D)&\rightarrow H^{-\frac{1}{2}}(\partial D)\nonumber\\
\phi& \mapsto\mathcal{K}^{*}_{\partial D}[\phi](\boldsymbol{x})=\int_{\partial D} \frac{\partial G(\boldsymbol{x},\boldsymbol{y})}{\partial \boldsymbol{\nu}_{\boldsymbol{x}}}\phi(\boldsymbol{y})\mathrm{d}s(\boldsymbol{y}).\label{OP:K}
\end{align}
Similarly, for a vectorial density $\boldsymbol{\phi}\in \mathbf{H}^{-\frac{1}{2}}_t(\partial D)$, the vectorial single-layer potential operator $\vec{\mathcal{S}}_{\partial D}$ is defined by 
\begin{align}\label{OP:VA}
\vec{\mathcal{S}}_{\partial D}:{H}^{-\frac{1}{2}}(\partial D)^3&\rightarrow H^{\frac{1}{2}}(\partial D)^3\text{ or } H^1_{\mathrm{loc}}(\mathbb{R}^3)^3\nonumber\\
\boldsymbol{\xi}&\mapsto\vec{\mathcal{S}}_{\partial D}[\boldsymbol{\xi}](\boldsymbol{x})=\int_{\partial D} G(\boldsymbol{x},\boldsymbol{y})\boldsymbol{\xi}(\boldsymbol{y})\mathrm{d}s(\boldsymbol{y}).
\end{align}

We now recall the  fundamental properties of the operators $\mathcal{S}_{\partial D},\vec{\mathcal{S}}_{\partial D},$ and $\mathcal{K}^{*}_{\partial D}$, as established in  \cite{Ammari2007polarization,Ammari2016Mathematical}.

\begin{lemm}\label{le:basic}
Let $\mathcal{I}$ denote the identity operator. The following properties hold:
\begin{enumerate}

\item[\rm{(1)}] The operator $\mathcal{S}_{\partial D}:H^{-\frac{1}{2}}(\partial D)\rightarrow H^{\frac{1}{2}}(\partial D)$ is bounded and invertible. Moreover, for any density function $\phi\in L^2(\partial D)$,  it satisfies the following  jump relation:
\begin{align*}
(\frac{\partial }{\partial \boldsymbol{\nu}}\mathcal{S}_{\partial D}[\phi])^{\pm}&=(\pm\frac{1}{2}\mathcal{I}+\mathcal{K}^{*}_{\partial D})[\phi] \quad \text{on }\partial D,
\end{align*}
where, for $ \boldsymbol{x}\in \partial D$,
\begin{align*}
(\frac{\partial }{\partial \boldsymbol{\nu}}\mathcal{S}_{\partial D}[\phi])^{\pm}(\boldsymbol{x})&:=\lim_{t\rightarrow 0^{+}}\nabla\mathcal{S}_{\partial D}[\phi](\boldsymbol{x}\pm t \boldsymbol{\nu}_{\boldsymbol{x}})\cdot\boldsymbol{\nu}_{\boldsymbol{x}}
\end{align*}
provided  the limit exists.

\item[\rm{(2)}]  The operator $\vec{\mathcal{S}}_{\partial D}:{H}^{-\frac{1}{2}}(\partial D)^3\rightarrow H^{\frac{1}{2}}(\partial D)^3$ is bounded and  the following jump formula holds: for any $\boldsymbol{\phi}\in \mathbf{H}^{-\frac{1}{2}}_t(\mathrm{div},\partial D)$,
\begin{align}\label{SS:jump}
(\boldsymbol{\nu}\times\nabla\times \vec{\mathcal{S}}_{\partial D}[\boldsymbol{\phi}])^{\pm}&=(\mp\frac{1}{2}\mathcal{I}+\mathcal{M}_{\partial D})[\boldsymbol{\phi}]\quad\text{on }\partial D,
\end{align}
where, for $ \boldsymbol{x}\in \partial D$,
\begin{align*}
(\boldsymbol{\nu}_{\boldsymbol{x}}\times\nabla\times \vec{\mathcal{S}}_{\partial D}[\boldsymbol{\phi}])^{\pm}(\boldsymbol{x})&:=\lim_{t\rightarrow 0^{+}}\boldsymbol{\nu}_{\boldsymbol{x}}\times\nabla\times \vec{\mathcal{S}}_{\partial D}[\boldsymbol{\phi}](\boldsymbol{x}\pm t \boldsymbol{\nu}_{\boldsymbol{x}})
\end{align*}
provided the limit exists.

\item[\rm{(3)}]  The operator $\mathcal{K}^*_{\partial D}:H^{-\frac{1}{2}}(\partial D)\rightarrow H^{-\frac{1}{2}}(\partial D)$ is bounded and self-adjoint on $H^{-\frac{1}{2}}(\partial D)$ with respect to the inner product $\langle\cdot,\cdot\rangle_{\mathcal{S}_{\partial D};H^{-\frac{1}{2}}(\partial D)}$ given by \eqref{inner:S}.
\end{enumerate}
\end{lemm}

We proceed to define the following bounded boundary integral operators:
\begin{align}
\mathcal{N}_{\partial D}:\mathbf{H}^{-\frac{1}{2}}_t(\mathrm{curl},\partial D)&\rightarrow \mathbf{H}^{-\frac{1}{2}}_t(\mathrm{div},\partial D)\nonumber\\
\boldsymbol{\phi}&\mapsto\mathcal{N}_{\partial D}[\boldsymbol{\phi}](\boldsymbol{x})
=\boldsymbol{\nu}_{\boldsymbol{x}}\times\nabla\times\nabla\times\int_{\partial D}G(\boldsymbol{x},\boldsymbol{y})\boldsymbol{\nu}_{\boldsymbol{y}}\times\boldsymbol{\phi}(\boldsymbol{y})\mathrm{d}s(\boldsymbol{y});\label{OP:N}\\
\mathcal{Q}_{\partial D}:\mathbf{H}^{-\frac{1}{2}}_t(\mathrm{div},\partial D)&\rightarrow \mathbf{H}^{-\frac{1}{2}}_t(\mathrm{curl},\partial D)\nonumber\\
\boldsymbol{\phi}&\mapsto\mathcal{Q}_{\partial D}[\boldsymbol{\phi}](\boldsymbol{x})
=-\boldsymbol{\nu}_{\boldsymbol{x}}\times\left(\boldsymbol{\nu}_{\boldsymbol{x}}\times\nabla\times\nabla\times\int_{\partial D}G(\boldsymbol{x},\boldsymbol{y})\boldsymbol{\phi}(\boldsymbol{y})\mathrm{d}s(\boldsymbol{y})\right),\label{OP:P}
\end{align}
where the boundedness of $\mathcal{Q}_{\partial D}$ can be similarly established as the proof of the boundedness of $\mathcal{N}_{\partial D}$; see   \cite[p.242]{Colton1998Inverse}.

The remainder of this subsection aims at recalling some fundamental results.  Since $\Delta_{\partial D}: H^{\frac{3}{2}}_{\star}(\partial D)\rightarrow H^{-\frac{1}{2}}_{\star}(\partial D)$ is invertible, then we have the following Helmholtz decompositions (cf. \cite[Theorem 5.5]{Buffa2002Traces} and \cite[Lemma 2.2]{Ammari2016Mathematical}).

\begin{theo}[Helmholtz Decomposition]
The following decompositions   hold:
\begin{align}
\mathbf{H}^{-\frac{1}{2}}_t(\mathrm{div},\partial D)&=\nabla_{\partial D}(H^{\frac{3}{2}}_{\star}(\partial D))\oplus \overset{\longrightarrow}{\mathrm{curl}}_{\partial D}(H^{\frac{1}{2}}(\partial D)),\nonumber\\
\mathbf{H}^{-\frac{1}{2}}_t(\mathrm{curl},\partial D)&=\overset{\longrightarrow}{\mathrm{curl}}_{\partial D}(H^{\frac{3}{2}}_{\star}(\partial D))\oplus \nabla_{\partial D}(H^{\frac{1}{2}}(\partial D)).\label{DE:curl}
\end{align}
\end{theo}
In particular, for any $\boldsymbol{\alpha}\in\nabla_{\partial D}(H^{\frac{1}{2}}(\partial D))$ and $\boldsymbol{\xi}\in\overset{\longrightarrow}{\mathrm{curl}}_{\partial D}(H^{\frac{1}{2}}(\partial D))$, there exist $X,V\in H^{\frac{1}{2}}(\partial D)$ such that
\begin{align*}
\boldsymbol{\alpha}=\nabla_{\partial D}X,\quad\boldsymbol{\xi}=\overset{\longrightarrow}{\mathrm{curl}}_{\partial D}V.
\end{align*} 
These subspaces are equipped with Sobolev norms defined through \eqref{norm:div} and \eqref{norm:curl} as follows:
\begin{align}
\|\boldsymbol{\alpha}\|_{\nabla_{\partial D}(H^{\frac{1}{2}}(\partial D))}&:=\|\nabla_{\partial D}X\|_{H^{-\frac{1}{2}}(\partial D)},\label{space:gradient}\\
\|\boldsymbol{\xi}\|_{\overset{\longrightarrow}{\mathrm{curl}}_{\partial D}(H^{\frac{1}{2}}(\partial D))}&:=\|\overset{\longrightarrow}{\mathrm{curl}}_{\partial D}V\|_{H^{-\frac{1}{2}}(\partial D)}.\label{space:curl}
\end{align}

This following lemma collects essential results from \cite[p.242]{Colton1998Inverse} and \cite[Proposition 4.3; Proposition 4.7]{Ammari2016Surface}.

\begin{lemm}\label{le:compact}
The following fundamental facts hold:

\begin{enumerate}

\item[\rm{(1)}]  For $\boldsymbol{\phi}\in \mathbf{H}^{-\frac{1}{2}}_t(\mathrm{div},\partial D)$,
\begin{align}
\nabla\cdot \vec{\mathcal{S}}_{\partial D}[\boldsymbol{\phi}]&=\mathcal{S}_{\partial D}[\nabla_{\partial D}\cdot\boldsymbol{\phi}]\quad\quad \text{in } \mathbb{R}^3,\label{I:indentity1}\\
\nabla_{\partial D}\cdot \mathcal{M}_{\partial D}[\boldsymbol{\phi}]&=-\mathcal{K}^{*}_{\partial D}[\nabla_{\partial D}\cdot\boldsymbol{\phi}]\quad\text{on }\partial D.\label{I:indentity3}
\end{align}

\item[\rm{(2)}]  For $\phi\in H^{\frac{1}{2}}(\partial D)$,
\begin{align}
\mathcal{M}^*_{\partial D}[\nabla_{\partial D} \phi]&=-\nabla_{\partial D}\mathcal{K}_{\partial D}[\phi]\quad\text{on } \partial D,\nonumber\\
\mathcal{M}_{\partial D}[\overset{\longrightarrow}{\mathrm{curl}}_{\partial D}\phi]&=\overset{\longrightarrow}{\mathrm{curl}}_{\partial D}\mathcal{K}_{\partial D}[\phi]\quad\text{on }\partial D,\label{I:indentity5}
\end{align}
where $\mathcal{K}_{\partial D}$ is the $L^{2}$-adjoint of $\mathcal{K}^{*}_{\partial D}$ defined in \eqref{OP:K}.

\item[\rm{(3)}] The operator $\mathcal{M}_{\partial D}: \mathbf{H}^{-\frac12}_t(\mathrm{div},\partial D)\rightarrow \mathbf{H}^{-\frac12}_t(\mathrm{div},\partial D)$ is compact. Its adjoint $\mathcal{M}^*_{\partial D}:\mathbf{H}^{-\frac12}_t(\mathrm{curl},\partial D)\rightarrow \mathbf{H}^{-\frac12}_t(\mathrm{curl},\partial D)$ is compact as well.

\item[\rm{(4)}]  Let $\sigma(\mathcal{M}_{\partial D})$ denote the spectrum of 
          $\mathcal{M}_{\partial D}$. Then
\begin{align*}
\sigma(\mathcal{M}_{\partial D})=(-\sigma(\mathcal{K}^*_{\partial D})\cup\sigma(\mathcal{K}^*_{\partial D}))\backslash\{1/2\}.
\end{align*}
\end{enumerate}
\end{lemm}

\subsection{Fundamental Properties of Boundary Integral Operators}

In this subsection, we  analyze the invertibility of the boundary integral operators $\mathcal{N}_{\partial D}$ and $\mathcal{Q}_{\partial D}$ defined in \eqref{OP:N} and \eqref{OP:P}, respectively.
These operators are essential for the subsequent symmetrization and spectral decomposition  of the MNP operator and its adjoint.

\begin{lemm}\label{le:Ker}
The operators 
\begin{align*}
\mathcal{N}_{\partial D}&:\mathbf{H}^{-\frac{1}{2}}_t(\mathrm{curl},\partial D)\rightarrow \overset{\longrightarrow}{\mathrm{curl}}_{\partial D}(H^{\frac{1}{2}}(\partial D)),\\
\mathcal{Q}_{\partial D}&:\mathbf{H}^{-\frac{1}{2}}_t(\mathrm{div},\partial D)\rightarrow \nabla_{\partial D}(H^{\frac{1}{2}}(\partial D))
\end{align*}
satisfy the following properties:

\begin{enumerate}
\item[\rm{(1)}]  (Self-adjointness) Both $\mathcal{N}_{\partial D}$ and $\mathcal{Q}_{\partial D}$
 are self-adjoint with respect to the bilinear form $\eqref{eq:bilinear}$.

\item[\rm{(2)}]   (Kernel Characterization)  Let $\mathrm{Ker}(\mathcal{N}_{\partial D})$ and $\mathrm{Ker}(\mathcal{Q}_{\partial D})$ denote  the kernels of $\mathcal{N}_{\partial D}$ and of $\mathcal{Q}_{\partial D}$ acting on $\mathbf{H}^{-\frac{1}{2}}_t(\mathrm{curl},\partial D)$ and $\mathbf{H}^{-\frac{1}{2}}_t(\mathrm{div},\partial D)$, respectively. The kernels of these operators satisfy:
\begin{align}
\mathrm{Ker}(\mathcal{N}_{\partial D})&=\nabla_{\partial D}(H^{\frac{1}{2}}(\partial D)),\label{ker:N}\\
 \mathrm{Ker}(\mathcal{Q}_{\partial D})&=\overset{\longrightarrow}{\mathrm{curl}}_{\partial D}(H^{\frac{1}{2}}(\partial D)).\label{ker:P}
\end{align}
\end{enumerate}
\end{lemm}
\begin{proof}

Using  the definitions of $\mathcal{N}_{\partial D}$ and $\mathcal{Q}_{\partial D}$ together  with identity \eqref{I:indentity1}, we obtain for any   $\boldsymbol{g}\in \mathbf{H}^{-\frac{1}{2}}_t(\mathrm{curl},\partial D)$ and $\boldsymbol{f}\in \mathbf{H}^{-\frac{1}{2}}_t(\mathrm{div},\partial D)$,
\begin{align}
\mathcal{N}_{\partial D}[\boldsymbol{g}]&=\boldsymbol{\nu}\times\nabla\nabla\cdot \vec{\mathcal{S}}_{\partial D}[\boldsymbol{\nu}\times\boldsymbol{g}]\nonumber\\
&=\boldsymbol{\nu}\times\nabla\mathcal{S}_{\partial D}[\nabla_{\partial D}\cdot(\boldsymbol{\nu}\times\boldsymbol{g})]\nonumber\\
&=\boldsymbol{\nu}\times(\nabla_{\partial D}\mathcal{S}_{\partial D}[\nabla_{\partial D}\cdot(\boldsymbol{\nu}\times\boldsymbol{g})]+\frac{\partial}{\partial\boldsymbol{\nu}}\mathcal{S}_{\partial D}[\nabla_{\partial D}\cdot(\boldsymbol{\nu}\times\boldsymbol{g})]\boldsymbol{\nu})\nonumber\\
&=\overset{\longrightarrow}{\mathrm{curl}}_{\partial D}\mathcal{S}_{\partial D}[\mathrm{curl}_{\partial D}\boldsymbol{g}],\label{IN:N}
\end{align}
and
\begin{align}
\mathcal{Q}_{\partial D}[\boldsymbol{f}]&=-\boldsymbol{\nu}\times(\boldsymbol{\nu}\times\nabla\nabla\cdot \vec{\mathcal{S}}_{\partial D}[\boldsymbol{f}])\nonumber\\
&=-\boldsymbol{\nu}\times(\boldsymbol{\nu}\times\nabla\mathcal{S}_{\partial D}[\nabla_{\partial D}\cdot\boldsymbol{f}])\nonumber\\
&=-\boldsymbol{\nu}\times(\boldsymbol{\nu}\times(\nabla_{\partial D}\mathcal{S}_{\partial D}[\nabla_{\partial D}\cdot\boldsymbol{f}]+\frac{\partial}{\partial\boldsymbol{\nu}}\mathcal{S}_{\partial D}[\nabla_{\partial D}\cdot\boldsymbol{f}]\boldsymbol{\nu}))\nonumber\\
&=\nabla_{\partial D}\mathcal{S}_{\partial D}[\nabla_{\partial D}\cdot\boldsymbol{f}].\label{In:P}
\end{align}
Observe that
\begin{align*}
\nabla_{\partial D}\cdot\mathcal{N}_{\partial D}[\boldsymbol g]=0,\quad\mathrm{curl}_{\partial D}\mathcal{Q}_{\partial D}[\boldsymbol f]=0.
\end{align*}
From the kernel characterizations
\begin{align*}
\mathrm{Ker}(\nabla_{\partial D}\cdot)=\overset{\longrightarrow}{\mathrm{curl}}_{\partial D}(H^{\frac{1}{2}}(\partial D)),\quad \mathrm{Ker}(\mathrm{curl}_{\partial D})=\nabla_{\partial D}(H^{\frac{1}{2}}(\partial D))
\end{align*}
established in  \cite[Theorem 5.1; Corollary 5.3]{Buffa2002Traces}, it follows that
\begin{align*}
\mathcal{N}_{\partial D}[\boldsymbol g]\in\overset{\longrightarrow}{\mathrm{curl}}_{\partial D}(H^{\frac{1}{2}}(\partial D)),\quad\mathcal{Q}_{\partial D}[\boldsymbol f]\in\nabla_{\partial D}(H^{\frac{1}{2}}(\partial D)).
\end{align*}

We now verify property $\mathrm{(1)}$. The self-adjointness of $\mathcal{N}_{\partial D}$ and $\mathcal{Q}_{\partial D}$ follows directly from representations \eqref{IN:N} and \eqref{In:P} combined with integration by parts.

Next, we establish the kernel identity \eqref{ker:N}.  Assume that there exists $\boldsymbol{g}\in \mathbf{H}^{-\frac{1}{2}}_t(\mathrm{curl},\partial D)$ such that $\mathcal{N}_{\partial D}[\boldsymbol{g}]=0$. By the Helmholtz decomposition \eqref{DE:curl}, we can write
\begin{align*}
\boldsymbol{g}\stackrel{\eqref{DE:curl}}{=}\nabla_{\partial D}X+\overset{\longrightarrow}{\mathrm{curl}}_{\partial D}Y,
\end{align*}
where $X\in H^{\frac{1}{2}}(\partial D)$ and $Y\in H^{\frac{3}{2}}_{\star}(\partial D)$. Consequently, we obtain
\begin{align*}
\mathcal{N}_{\partial D}[\boldsymbol{g}]
\stackrel{\eqref{IN:N}}{=}\overset{\longrightarrow}{\mathrm{curl}}_{\partial D}\mathcal{S}_{\partial D}[\mathrm{curl}_{\partial D}(\nabla_{\partial D}X+\overset{\longrightarrow}{\mathrm{curl}}_{\partial D}Y)]=\overset{\longrightarrow}{\mathrm{curl}}_{\partial D}\mathcal{S}_{\partial D}[\Delta_{\partial D}Y]=0,
\end{align*}
where $\Delta_{\partial D}Y\in H^{-\frac{1}{2}}_\star(\partial D)$. Moreover,  it follows from \cite[Corollary 3.7]{Buffa2002Traces} that $\mathrm{Ker}(\overset{\longrightarrow}{\mathrm{curl}}_{\partial D})=\mathbb{R}$. This implies that
\begin{align*}
\mathcal{S}_{\partial D}[\Delta_{\partial D}Y]=C\in\mathbb{R}.
\end{align*}
Since $\Delta_{\partial D}Y\in H^{-\frac{1}{2}}_\star(\partial D)$ has zero mean, the single‑layer potential of a zero‑mean  function is itself zero‑mean, which forces $C=0$. Thus 
\begin{align*}
\mathcal{S}_{\partial D}[\Delta_{\partial D}Y]=0,
\end{align*}
and the invertibility of $\mathcal{S}_{\partial D}$ implies $\Delta_{\partial D}Y=0$.  Integration by parts then yields
$\overset{\longrightarrow}{\mathrm{curl}}_{\partial D}Y=0$.

From the above analysis, if  $\boldsymbol{g}\in \mathbf{H}^{-\frac{1}{2}}_t(\mathrm{curl},\partial D)$ such that $\mathcal{N}_{\partial D}[\boldsymbol{g}]=0$, then
\begin{align*}
\mathcal{S}_{\partial D}[\mathrm{curl}_{\partial D}\boldsymbol{g}]=0,
\end{align*}
which implies  $\mathrm{curl}_{\partial D}\boldsymbol{g}=0$. By \cite[Theorem 5.1]{Buffa2002Traces},  this 
means $\boldsymbol{g}\in\nabla_{\partial D}(H^{\frac12}(\partial D))$, proving \eqref{ker:N}.

The identity \eqref{ker:P} is proved analogously. Observe that
\begin{align*}
\mathrm{Ker}(\nabla_{\partial D})=\mathbb{R}
\end{align*}
according to \cite[Corollary 3.7]{Buffa2002Traces}. By proceeding a similar process as the proof of \eqref{ker:N}, we establish  \eqref{ker:P}.
\end{proof}

We now establish the invertibility of the operators $\mathcal{N}_{\partial D}$ and $\mathcal{Q}_{\partial D}$.

\begin{prop}\label{le:N-inverse}
The following properties hold:

\begin{enumerate}
\item[\rm{(1)}]  The operator $\mathcal{N}_{\partial D}:\overset{\longrightarrow}{\mathrm{curl}}_{\partial D}(H^{\frac{3}{2}}_{\star}(\partial D))\rightarrow \overset{\longrightarrow}{\mathrm{curl}}_{\partial D}(H^{\frac{1}{2}}(\partial D))$  admits a bounded inverse
\begin{align*}
\mathcal{N}^{-1}_{\partial D}:\overset{\longrightarrow}{\mathrm{curl}}_{\partial D}(H^{\frac{1}{2}}(\partial D))\rightarrow\overset{\longrightarrow}{\mathrm{curl}}_{\partial D}(H^{\frac{3}{2}}_{\star}(\partial D)).
\end{align*}

\item[\rm{(2)}]  The operator $\mathcal{Q}_{\partial D}:\nabla_{\partial D}(H^{\frac{3}{2}}_{\star}(\partial D))\rightarrow \nabla_{\partial D}(H^{\frac{1}{2}}(\partial D))$ possesses a bounded inverse satisfying
\begin{align*}
\mathcal{Q}^{-1}_{\partial D}&: \nabla_{\partial D}(H^{\frac{1}{2}}(\partial D))\rightarrow\nabla_{\partial D}(H^{\frac{3}{2}}_{\star}(\partial D)).
\end{align*}
\end{enumerate}
\end{prop}
\begin{proof}
We first prove part $\mathrm{(1)}$ of Proposition \ref{le:N-inverse}.   Given $\boldsymbol{g}=\overset{\longrightarrow}{\mathrm{curl}}_{\partial D}Y$ with $Y\in H^{\frac32}_{\star}(\partial D)$,  we have 
\begin{align*}
\mathrm{curl}_{\partial D}\boldsymbol{g}=\mathrm{curl}_{\partial D}\overset{\longrightarrow}{\mathrm{curl}}_{\partial D}Y=-\Delta_{\partial D}Y\in H^{-\frac{1}{2}}_\star(\partial D).
\end{align*}
Since  $\Delta_{\partial D}: H^{\frac{3}{2}}_{\star}(\partial D)\rightarrow H^{-\frac{1}{2}}_{\star}(\partial D)$ is injective and surjective,  the map
\begin{align*}
\mathrm{curl}_{\partial D}:\overset{\longrightarrow}{\mathrm{curl}}_{\partial D}(H^{\frac32}_{\star}(\partial D))\rightarrow H^{-\frac{1}{2}}_\star(\partial D)
\end{align*}
is surjective.  Moreover, according to \cite{Verchota1984Layer}, the single-layer operator
\begin{align*}
\mathcal{S}_{\partial D}: H^{-\frac{1}{2}}_\star(\partial D)\rightarrow H^{\frac{1}{2}}_{\star}(\partial D)
\end{align*}
is also surjective. It is straightforward that
\begin{align*}
 \overset{\longrightarrow}{\mathrm{curl}}_{\partial D}:H^{\frac{1}{2}}_\star(\partial D)\rightarrow\overset{\longrightarrow}{\mathrm{curl}}_{\partial D}(H^{\frac{1}{2}}_\star(\partial D))=\overset{\longrightarrow}{\mathrm{curl}}_{\partial D}(H^{\frac{1}{2}}(\partial D))
\end{align*}
is surjective. 
In addition, Lemma \ref{le:Ker} guarantees that the boundary operator
\begin{align*}
\mathcal{N}_{\partial D}:\overset{\longrightarrow}{\mathrm{curl}}_{\partial D}(H^{\frac32}_{\star}(\partial D))\rightarrow \overset{\longrightarrow}{\mathrm{curl}}_{\partial D}(H^{\frac{1}{2}}(\partial D))
\end{align*}
is injective. Hence, part $\mathrm{(1)}$  of Proposition \ref{le:N-inverse} holds.

On the other hand, for any $\boldsymbol{f} \in \nabla_{\partial D}(H^{\frac{3}{2}}_\star(\partial D))$, an analogous argument to the one above demonstrates the invertibility of the operator $\mathcal{Q}_{\partial D}$, that is, we have part $\mathrm{(2)}$ in Proposition \ref{le:N-inverse}.
\end{proof}

\section{Symmetrization and spectral decomposition for the static MNP Operator and its adjoint}\label{sec:4}
The primary objective of this section is to introduce a new bilinear form  and verify that it indeed defines an inner product. We subsequently demonstrate that the static MNP operator can be symmetrized under this new inner product, thereby completing the proof of Theorem \ref{th:symmetrisation1}. Furthermore, we establish the equivalence between the norm induced by this new inner product and the standard Sobolev norm. This equivalence enables the study of  compactness properties of the static MNP operator, which in turn leads to the completion of Theorem \ref{th:decomposition1}. Following similar arguments, we extend our analysis to the corresponding adjoint operator. This completes the proofs of Theorems \ref{th:symmetrisation2} and \ref{th:decomposition2}.

\subsection{Symmetrization for the static MNP operator and its adjoint}

We begin by introducing a weighted bilinear form in $\overset{\longrightarrow}{\mathrm{curl}}_{\partial D}(H^{\frac{1}{2}}(\partial D))$.  Let us define a  bilinear form $\langle\cdot,\cdot\rangle_{\mathcal{N}^{-1}_{\partial D};\overset{\longrightarrow}{\mathrm{curl}}_{\partial D}(H^{\frac{1}{2}}(\partial D))}$ on $\overset{\longrightarrow}{\mathrm{curl}}_{\partial D}(H^{\frac{1}{2}}(\partial D))$ by
\begin{align}\label{E:inner}
\langle\boldsymbol \xi,\boldsymbol \zeta\rangle_{\mathcal{N}^{-1}_{\partial D};\overset{\longrightarrow}{\mathrm{curl}}_{\partial D}(H^{\frac{1}{2}}(\partial D))}:=\int_{\partial D}\boldsymbol{\xi}(\boldsymbol{x})\cdot\mathcal{N}^{-1}_{\partial D}[\boldsymbol {\zeta}](\boldsymbol{x})\mathrm{d}s(\boldsymbol{x})
\end{align}
for any $\boldsymbol \xi,\boldsymbol \zeta\in\overset{\longrightarrow}{\mathrm{curl}}_{\partial D}(H^{\frac{1}{2}}(\partial D))$.

We first verify   the positive definiteness of  the bilinear form $\langle\cdot,\cdot\rangle_{\mathcal{N}^{-1}_{\partial D};\overset{\longrightarrow}{\mathrm{curl}}_{\partial D}(H^{\frac{1}{2}}(\partial D))}$.

\begin{lemm}\label{le:positive}
The bilinear form given by \eqref{E:inner} is positive-definite on $\overset{\longrightarrow}{\mathrm{curl}}_{\partial D}(H^{\frac{1}{2}}(\partial D))$.
\end{lemm}
\begin{proof}
For any $\boldsymbol \xi\in \overset{\longrightarrow}{\mathrm{curl}}_{\partial D}(H^{\frac{1}{2}}(\partial D))$, there exists   $\boldsymbol {\eta}\in\overset{\longrightarrow}{\mathrm{curl}}_{\partial D}(H^{\frac32}_{\star}(\partial D))$ such that $\mathcal{N}^{-1}_{\partial D}[\boldsymbol{\xi}]=\boldsymbol {\eta}$. That is, $\boldsymbol{\xi}=\mathcal{N}_{\partial D}[\boldsymbol {\eta}]$. Moreover, we can write $\boldsymbol{\eta}=\overset{\longrightarrow}{\mathrm{curl}}_{\partial D} V$ with $V\in H^{\frac32}_{\star}(\partial D)$.
It follows from \eqref{IN:N} and  integration by parts that
\begin{align*}
\langle\boldsymbol \xi,\boldsymbol \xi\rangle_{\mathcal{N}^{-1}_{\partial D};\overset{\longrightarrow}{\mathrm{curl}}_{\partial D}(H^{\frac{1}{2}}(\partial D))}&=\int_{\partial D}\mathcal{N}_{\partial D}[\boldsymbol{\eta}]\cdot\boldsymbol {\eta}\mathrm{d}s(\boldsymbol{x}) \nonumber\\
&=-\int_{\partial D}(\mathrm{curl}_{\partial D}\boldsymbol{\eta})\mathcal{S}_{\partial D}[\mathrm{curl}_{\partial D}\boldsymbol{\eta}]\mathrm{d}s(\boldsymbol{x}) \nonumber\\
&=-\int_{\partial D}(\Delta_{\partial D}V)\mathcal{S}_{\partial D}[\Delta_{\partial D}V]\mathrm{d}s(\boldsymbol{x}).
\end{align*}
If $\langle\boldsymbol \xi,\boldsymbol \xi\rangle_{\mathcal{N}^{-1}_{\partial D};\overset{\longrightarrow}{\mathrm{curl}}_{\partial D}(H^{\frac{1}{2}}(\partial D))}=0$, then the positive-definiteness of  $-\mathcal{S}_{\partial D}$ implies $\Delta_{\partial D}V=0$. Then using the invertibility of $\Delta_{\partial D}$ yields $V=0$,  
which leads to
\begin{align*}
\boldsymbol{\xi}=\mathcal{N}_{\partial D}[\boldsymbol {\eta}]=\mathcal{N}_{\partial D}[\overset{\longrightarrow}{\mathrm{curl}}_{\partial D} V]=0.
\end{align*}
Hence, the bilinear form $\langle\cdot,\cdot\rangle_{\mathcal{N}^{-1}_{\partial D};\overset{\longrightarrow}{\mathrm{curl}}_{\partial D}(H^{\frac{1}{2}}(\partial D))}$ is positive-definite and the proof is complete.
\end{proof}

By Lemma \ref{le:positive}, the bilinear form $\langle\cdot,\cdot\rangle_{\mathcal{N}^{-1}_{\partial D};\overset{\longrightarrow}{\mathrm{curl}}_{\partial D}(H^{\frac{1}{2}}(\partial D))}$ is an inner product on  $\overset{\longrightarrow}{\mathrm{curl}}_{\partial D}(H^{\frac{1}{2}}(\partial D))$. The norm induced by this inner product is denoted by $\|\cdot\|_{\mathcal{N}^{-1}_{\partial D};\overset{\longrightarrow}{\mathrm{curl}}_{\partial D}(H^{\frac{1}{2}}(\partial D))}$. Furthermore, we have the following additional property for $\mathcal{N}_{\partial D}$.

\begin{lemm}
In the inner product $\langle\cdot,\cdot\rangle_{\mathcal{N}^{-1}_{\partial D};\overset{\longrightarrow}{\mathrm{curl}}_{\partial D}(H^{\frac{1}{2}}(\partial D))}$, the operator $\mathcal{N}_{\partial D}$ is  self-adjoint.
\end{lemm}
\begin{proof}
The self-adjointness follows immediately from \eqref{IN:N}, \eqref{E:inner}, and Proposition \ref{le:N-inverse} (1).
\end{proof}

Recall the following Calder\'{o}n identity on  $\mathcal{K}^*_{\partial D}$:
\begin{align}
\mathcal{K}_{\partial D}\mathcal{S}_{\partial D}&=\mathcal{S}_{\partial D}\mathcal{K}^*_{\partial D}.\label{K:Cald1}
\end{align}
With this identity established, we now proceed to complete the proof of Theorem \ref{th:symmetrisation1}.

\begin{proof}[Proof of Theorem \ref{th:symmetrisation1}]
We begin by proving part $\mathrm{(2)}$ of Theorem \ref{th:symmetrisation1}, which asserts the validity of formula \eqref{KK:Cald3}.

On the one hand, using \eqref{eq:adjoint}, \eqref{eq:DC}, \eqref{I:indentity3}, and \eqref{IN:N} yields that for any $\boldsymbol{f}\in\overset{\longrightarrow}{\mathrm{curl}}_{\partial D}(H^{\frac32}_{\star}(\partial D))$,
\begin{align*}
\mathcal{N}_{\partial D}\mathcal{M}^*_{\partial D}[\boldsymbol{f}]
&=\overset{\longrightarrow}{\mathrm{curl}}_{\partial D}(\mathcal{S}_{\partial D}[\nabla_{\partial D}\cdot\mathcal{M}_{\partial D}[\boldsymbol{\nu}\times\boldsymbol{f}]])\\
&=\overset{\longrightarrow}{\mathrm{curl}}_{\partial D}(\mathcal{S}_{\partial D}K^*_{\partial D}[\mathrm{curl}_{\partial D}\boldsymbol{f}]).
\end{align*}
On the other hand, it follows from \eqref{IN:N}, \eqref{I:indentity5}, and \eqref{K:Cald1} that for any $\boldsymbol{f}\in\overset{\longrightarrow}{\mathrm{curl}}_{\partial D}(H^{\frac32}_{\star}(\partial D))$,
\begin{align*}
\mathcal{M}_{\partial D}\mathcal{N}_{\partial D}[\boldsymbol{f}]&=\mathcal{M}_{\partial D}[\overset{\longrightarrow}{\mathrm{curl}}_{\partial D}\mathcal{S}_{\partial D}[\mathrm{curl}_{\partial D}\boldsymbol{f}]]\\
&=\overset{\longrightarrow}{\mathrm{curl}}_{\partial D}(\mathcal{K}_{\partial D}\mathcal{S}_{\partial D}[\mathrm{curl}_{\partial D}\boldsymbol{f}])\\
&=\overset{\longrightarrow}{\mathrm{curl}}_{\partial D}(\mathcal{S}_{\partial D}\mathcal{K}^*_{\partial D}[\mathrm{curl}_{\partial D}\boldsymbol{f}]).
\end{align*}
Hence, we can derive \eqref{KK:Cald3}.

It remains to establish the self-adjointness of $\mathcal{M}_{\partial D}$, as stated in part $\mathrm{(2)}$ of Theorem \ref{th:symmetrisation1}. Notice that for any $\boldsymbol{\xi}\in\overset{\longrightarrow}{\mathrm{curl}}_{\partial D}(H^{\frac{1}{2}}(\partial D))$,
\begin{align*}
\nabla_{\partial D}\cdot\mathcal{M}_{\partial D}[\boldsymbol{\xi}]\stackrel{\eqref{I:indentity3}}{=}-\mathcal{K}^*_{\partial D}[\nabla_{\partial D}\cdot\boldsymbol{\xi}]=0.
\end{align*}
Then using \cite[Corollary 5.3]{Buffa2002Traces} yields 
\begin{align*}
\mathcal{M}_{\partial D}[\boldsymbol{\xi}]\in\overset{\longrightarrow}{\mathrm{curl}}_{\partial D}(H^{\frac{1}{2}}(\partial D)).
\end{align*}
It follows from \eqref{E:inner} and \eqref{KK:Cald3} that
\begin{align}
\langle\mathcal{M}_{\partial D}[\boldsymbol\xi],\boldsymbol \zeta \rangle_{\mathcal{N}^{-1}_{\partial D};\overset{\longrightarrow}{\mathrm{curl}}_{\partial D}(H^{\frac{1}{2}}(\partial D))}&=\int_{\partial D}\mathcal{M}_{\partial D}[\boldsymbol{\xi}]\cdot\mathcal{N}^{-1}_{\partial D}[\boldsymbol \zeta]\mathrm{d}s(\boldsymbol{x})\nonumber\\
&=\int_{\partial D}\boldsymbol{\xi}\cdot\mathcal{M}^*_{\partial D}\mathcal{N}^{-1}_{\partial D}[\boldsymbol{\zeta}]\mathrm{d}s(\boldsymbol{x})\label{M:range}\\
&=\int_{\partial D}\boldsymbol{\xi}\cdot\mathcal{N}^{-1}_{\partial D}\mathcal{N}_{\partial D}\mathcal{M}^*_{\partial D}\mathcal{N}^{-1}_{\partial D}[\boldsymbol{\zeta}]\mathrm{d}s(\boldsymbol{x})\nonumber\\
&=\int_{\partial D}\boldsymbol{\xi}\cdot\mathcal{N}^{-1}_{\partial D}\mathcal{M}_{\partial D}\mathcal{N}_{\partial D}\mathcal{N}^{-1}_{\partial D}[\boldsymbol{\zeta}]\mathrm{d}s(\boldsymbol{x})\nonumber\\
&=\langle\boldsymbol\xi,\mathcal{M}_{\partial D}[\boldsymbol \zeta]\rangle_{\mathcal{N}^{-1}_{\partial D};\overset{\longrightarrow}{\mathrm{curl}}_{\partial D}(H^{\frac{1}{2}}(\partial D))}.\nonumber
\end{align}
Hence, $\mathcal{M}_{\partial D}$ is self-adjoint  with respect to  
$\langle\cdot,\cdot\rangle_{\mathcal{N}^{-1}_{\partial D};
\overset{\longrightarrow}{\mathrm{curl}}_{\partial D}(H^{\frac12}(\partial D))}$, which completes the proof of Theorem \ref{th:symmetrisation1}.
\end{proof}

\begin{rema}
According to \eqref{M:range}, the exact formulation of the adjoint operator $\mathcal{M}^*_{\partial D}$ is the projection of the restricted operator $\mathcal{M}^*_{\partial D}\Big|_{\overset{\longrightarrow}{\mathrm{curl}}_{\partial D}(H^{\frac32}_{\star}(\partial D))}$ onto the subspace $\overset{\longrightarrow}{\mathrm{curl}}_{\partial D}(H^{\frac32}_{\star}(\partial D))$.
\end{rema}
For any $\boldsymbol \alpha,\boldsymbol \beta\in\nabla_{\partial D}(H^{\frac{1}{2}}(\partial D))$, we proceed to define a bilinear form $\langle\cdot,\cdot\rangle_{\mathcal{Q}^{-1}_{\partial D};\nabla_{\partial D}(H^{\frac{1}{2}}(\partial D))}$ on $\nabla_{\partial D}(H^{\frac{1}{2}}(\partial D))$ by
\begin{align}\label{EE:inner}
\langle\boldsymbol \alpha,\boldsymbol \beta\rangle_{\mathcal{Q}^{-1}_{\partial D};\nabla_{\partial D}(H^{\frac{1}{2}}(\partial D))}:=\int_{\partial D}\boldsymbol{\alpha}(\boldsymbol{x})\cdot\mathcal{Q}^{-1}_{\partial D}[\boldsymbol {\beta}](\boldsymbol{x})\mathrm{d}s(\boldsymbol{x}).
\end{align}
By arguments analogous to those used for the curl‑trace space, one can study the 
symmetrization of the adjoint operator \(\mathcal{M}^*_{\partial D}\) on the gradient-trace space $\nabla_{\partial D}(H^{\frac{1}{2}}(\partial D))$.

We are now in a position to give the proof of Theorem \ref{th:symmetrisation2}.
\begin{proof}[Proof of Theorem \ref{th:symmetrisation2}]
Obviously, the bilinear form $\langle\cdot,\cdot\rangle_{\mathcal{Q}^{-1}_{\partial D};\nabla_{\partial D}(H^{\frac{1}{2}}(\partial D))}$  is positive-definite, hence an inner product 
on $\nabla_{\partial D}(H^{\frac{1}{2}}(\partial D))$. 
Repeating the steps of the proof of Theorem~\ref{th:symmetrisation1} with the 
obvious substitutions (replacing \(\mathcal{N}_{\partial D}\) by \(\mathcal{Q}_{\partial D}\), 
the curl‑trace spaces by gradient‑trace spaces, and using the identities 
\eqref{I:indentity3} and \eqref{I:indentity5}) yields the Calder\'{o}n identity 
\eqref{KK:Cald4} on \(\nabla_{\partial D}
\bigl(H^{\frac32}_{\star}(\partial D)\bigr)\). We can also  verify the symmetrization of
$\mathcal{M}^*_{\partial D}:\nabla_{\partial D}(H^{\frac{1}{2}}(\partial D))\rightarrow\nabla_{\partial D}(H^{\frac{1}{2}}(\partial D))$ with respect to the inner product $\langle\boldsymbol \cdot,\boldsymbol \cdot\rangle_{\mathcal{Q}^{-1}_{\partial D};\nabla_{\partial D}(H^{\frac{1}{2}}(\partial D))}$.  
Since the argument is entirely parallel to that of Theorem~\ref{th:symmetrisation1}, 
we omit the details for brevity. 
\end{proof}

\subsection{Spectral decomposition for the static MNP operator and its adjoint}
In this subsection, by utilizing the equivalence between the inner-product-induced norm $\|\cdot\|_{\mathcal{N}^{-1}_{\partial D};\overset{\longrightarrow}{\mathrm{curl}}_{\partial D}(H^{\frac{1}{2}}(\partial D))}$ and the standard Sobolev norm $\|\cdot\|_{\overset{\longrightarrow}{\mathrm{curl}}_{\partial D}(H^{\frac{1}{2}}(\partial D))}$, we analyze the compactness of  $\mathcal{M}_{\partial D}$ restricted to $\overset{\longrightarrow}{\mathrm{curl}}_{\partial D}(H^{\frac{1}{2}}(\partial D))$ and establish its spectral decomposition theorem. Parallel arguments are then developed for the spectral decomposition of $\mathcal{M}^*_{\partial D}$  restricted to $\nabla_{\partial D}(H^{\frac{1}{2}}(\partial D))$.

First, we establish the result on the equivalence of norms.

\begin{lemm}\label{le:equivalent}
The norm $\|\cdot\|_{\mathcal{N}^{-1}_{\partial D};\overset{\longrightarrow}{\mathrm{curl}}_{\partial D}(H^{\frac{1}{2}}(\partial D))}$ induced by the inner product $\langle\cdot,\cdot\rangle_{\mathcal{N}^{-1}_{\partial D};\overset{\longrightarrow}{\mathrm{curl}}_{\partial D}(H^{\frac{1}{2}}(\partial D))}$, defined in \eqref{E:inner}, is equivalent to the Sobelev norm  $\|\cdot\|_{\overset{\longrightarrow}{\mathrm{curl}}_{\partial D}(H^{\frac{1}{2}}(\partial D))}$ given by \eqref{space:curl}.
\end{lemm}
\begin{proof}
It is obvious that for any $\boldsymbol{\xi}\in \overset{\longrightarrow}{\mathrm{curl}}_{\partial D}(H^{\frac{1}{2}}(\partial D))$,
\begin{align*}
\|\boldsymbol{\xi}\|^2_{\mathcal{N}^{-1}_{\partial D};\overset{\longrightarrow}{\mathrm{curl}}_{\partial D}(H^{\frac{1}{2}}(\partial D))}&=|\int_{\partial D}\boldsymbol{\xi}\cdot\mathcal{N}^{-1}_{\partial D}[\boldsymbol {\xi}]\mathrm{d}s(\boldsymbol{x})|\\
&\leq\|\boldsymbol{\xi}\|_{\overset{\longrightarrow}{\mathrm{curl}}_{\partial D}(H^{\frac{1}{2}}(\partial D))}\|\mathcal{N}^{-1}_{\partial D}[\boldsymbol{\xi}]\|_{\overset{\longrightarrow}{\mathrm{curl}}_{\partial D}(H^{\frac{3}{2}}_{\star}(\partial D))}\\
&\leq C\|\boldsymbol{\xi}\|^2_{\overset{\longrightarrow}{\mathrm{curl}}_{\partial D}(H^{\frac{1}{2}}(\partial D))}.
\end{align*}

On the other hand, from Proposition \ref{le:N-inverse}, there is  $\boldsymbol{\eta}\in \overset{\longrightarrow}{\mathrm{curl}}_{\partial D}(H^{\frac{3}{2}}_{\star}(\partial D))$, with $\boldsymbol{\eta}=\overset{\longrightarrow}{\mathrm{curl}}_{\partial D}V$ and $V\in H^{\frac{3}{2}}_{\star}(\partial D)$, such that $\boldsymbol{\xi}=\mathcal{N}_{\partial D}[\boldsymbol{\boldsymbol{\eta}}]$. As a consequence, for some constant $C>0$,
\begin{align*}
\|\boldsymbol{\xi}\|_{\mathcal{N}^{-1}_{\partial D};\overset{\longrightarrow}{\mathrm{curl}}_{\partial D}(H^{\frac{1}{2}}(\partial D))}=-\int_{\partial D}(\Delta_{\partial D}V)\mathcal{S}_{\partial D}[\Delta_{\partial D}V]\mathrm{d}s(\boldsymbol{x})=\|\Delta_{\partial D}V\|_{\mathcal{S}_{\partial D};H^{-\frac{1}{2}}(\partial D)},
\end{align*}
which is equivalent  to the Sobolev norm in $H^{-\frac{1}{2}}(\partial D)$, namely, there exist two positive constants $C_1,C_2$ satisfying
\begin{align*}
C_1\|\Delta_{\partial D}V\|_{H^{-\frac{1}{2}}(\partial D)}\leq\|\boldsymbol{\xi}\|_{\mathcal{N}^{-1}_{\partial D};\overset{\longrightarrow}{\mathrm{curl}}_{\partial D}(H^{\frac{1}{2}}(\partial D))}\leq C_2\|\Delta_{\partial D}V\|_{H^{-\frac{1}{2}}(\partial D)}.
\end{align*}
In addition, using \eqref{space:curl} and \eqref{IN:N} yields that
\begin{align*}
\|\boldsymbol{\xi}\|_{\overset{\longrightarrow}{\mathrm{curl}}_{\partial D}(H^{\frac{1}{2}}(\partial D))}
&=\|\mathcal{N}_{\partial D}[\boldsymbol{\eta}]\|_{H^{-\frac{1}{2}}(\partial D)}\\
&=\|\overset{\longrightarrow}{\mathrm{curl}}\mathcal{S}_{\partial D}[\Delta _{\partial D}V]\|_{H^{-\frac{1}{2}}(\partial D)}\\
&\leq C\|\Delta_{\partial D}V\|_{H^{-\frac{1}{2}}(\partial D)},
\end{align*}
which implies the conclusion of this lemma and completes the corresponding proof.
\end{proof}

\begin{prop}\label{prop:com}
The operator 
\begin{align*}
\mathcal{M}_{\partial D}:\overset{\longrightarrow}{\mathrm{curl}}_{\partial D}(H^{\frac{1}{2}}(\partial D))\rightarrow \overset{\longrightarrow}{\mathrm{curl}}_{\partial D}(H^{\frac{1}{2}}(\partial D))
\end{align*}
is compact with respect to the norm  $\|\cdot\|_{\mathcal{N}^{-1}_{\partial D};\overset{\longrightarrow}{\mathrm{curl}}_{\partial D}(H^{\frac{1}{2}}(\partial D))}$.
\end{prop}
\begin{proof}
If we restrict $\mathcal{M}_{\partial D}$ on $\overset{\longrightarrow}{\mathrm{curl}}_{\partial D}(H^{\frac{1}{2}}(\partial D))$, then it is compact,
where we apply property $\mathrm{(3)}$ in Lemma \ref{le:compact}, the definition of the norm $\|\cdot\|_{\overset{\longrightarrow}{\mathrm{curl}}_{\partial D}(H^{\frac{1}{2}}(\partial D))}$ given by \eqref{space:curl}, and Lemma \ref{le:equivalent}.
\end{proof}

The following lemma establishes the fact that $\lambda$ is an eigenvalue of $\mathcal{M}_{\partial D}$ on $\overset{\longrightarrow}{\mathrm{curl}}_{\partial D}(H^{\frac{1}{2}}(\partial D))$ if and only if $\lambda$ is an eigenvalue of $\mathcal{K}^{*}_{\partial D}$ on $H^{-\frac{1}{2}}(\partial D)$.

\begin{lemm}\label{le:spectrum}
One has $\sigma(\mathcal{M}_{\partial D};\overset{\longrightarrow}{\mathrm{curl}}_{\partial D}(H^{\frac{1}{2}}(\partial D)))=\sigma(\mathcal{K}^*_{\partial D})\backslash\{1/2\}$.
\end{lemm}
\begin{proof}

Assume first that $\lambda\neq{1}/{2}$ is an eigenvalue of $\mathcal{K}^*_{\partial D}$.
By the self-adjointness of $\mathcal{K}^*_{\partial D}$ with respect to the inner product
 $\langle\cdot,\cdot\rangle_{\mathcal{S}_{\partial D};H^{-\frac{1}{2}}(\partial D)}$, $\lambda$ is also an eigenvalue of  
$\mathcal{K}_{\partial D}$. Therefore, there exists a nonzero $\theta\in H^{\frac12}(\partial D)$
such that
\begin{align*}
(\lambda\mathcal{I}-\mathcal{K}_{\partial D})[\theta]=0,
\end{align*}
which leads to
\begin{align}\label{E:equality}
(\lambda\mathcal{I}-\mathcal{M}_{\partial D})[\overset{\longrightarrow}{\mathrm{curl}}_{\partial D}\theta]\stackrel{\eqref{I:indentity5}}{=}\overset{\longrightarrow}{\mathrm{curl}}_{\partial D}(\lambda\mathcal{I}-\mathcal{K}_{\partial D})[\theta]=0.
\end{align}
If $\overset{\longrightarrow}{\mathrm{curl}}_{\partial D}\theta=0$, then Corollary 3.7 in \cite{Buffa2002Traces} implies that $\theta$ is constant, yielding a contradiction.

Conversely, if  $\lambda\neq{1}/{2}$ is an eigenvalue of $\mathcal{M}_{\partial D}$, then applying \eqref{E:equality} and together with the fact that $\mathrm{Ker}(\overset{\longrightarrow}{\mathrm{curl}}_{\partial D})=\mathbb{R}$ (recall \cite[Corollary 3.7]{Buffa2002Traces}), we deduce that 
\begin{align*}
(\lambda\mathcal{I}-\mathcal{K}_{\partial D})[\theta]=c,
\end{align*}
where $c$ is a constant. Since $\mathcal{K}_{\partial D}[1]={1}/{2}$ and  $\lambda\neq{1}/{2}$, we obtain
\begin{align*}
(\lambda\mathcal{I}-\mathcal{K}_{\partial D})\left[\theta+\frac{c}{\lambda-{1}/{2}}\right]=0.
\end{align*}
Finally, we recall property $\mathrm{(3)}$ in Lemma \ref{le:basic}, which completes the proof of the lemma.
\end{proof}

Furthermore, leveraging the inner product $\langle\cdot,\cdot\rangle_{\mathcal{N}^{-1}_{\partial D};\overset{\longrightarrow}{\mathrm{curl}}_{\partial D}(H^{\frac{1}{2}}(\partial D))}$ defined in \eqref{E:inner} and the invertibility of $\mathcal{N}_{\partial D}$ provided by Proposition \ref{le:N-inverse}, we introduce an inner product on $\overset{\longrightarrow}{\mathrm{curl}}_{\partial D}(H^{\frac32}_{\star}(\partial D))$. Specifically, for any $\boldsymbol f,\boldsymbol g\in\overset{\longrightarrow}{\mathrm{curl}}_{\partial D}(H^{\frac32}_{\star}(\partial D))$, set
\begin{align}\label{E:Product}
\langle\boldsymbol f,\boldsymbol g\rangle_{\mathcal{N}_{\partial D};\overset{\longrightarrow}{\mathrm{curl}}_{\partial D}(H^{\frac32}_{\star}(\partial D))}:&=
\langle\mathcal{N}_{\partial D}[\boldsymbol f],\mathcal{N}_{\partial D}[\boldsymbol g]\rangle_{\mathcal{N}^{-1}_{\partial D};\overset{\longrightarrow}{\mathrm{curl}}_{\partial D}(H^{\frac{1}{2}}(\partial D))}\\
&=
\int_{\partial D}\mathcal{N}_{\partial D}[\boldsymbol f](\boldsymbol{x})\cdot \boldsymbol g(\boldsymbol{x})\mathrm{d}s(\boldsymbol{x}).\nonumber
\end{align}
The induced norm is denoted by $\|\cdot\|_{\mathcal{N}_{\partial D};\overset{\longrightarrow}{\mathrm{curl}}_{\partial D}(H^{\frac32}_{\star}(\partial D))}$.

\begin{proof}[Proof of Theorem \ref{th:decomposition1}]

An immediate application of Theorem \ref{th:symmetrisation1} and Proposition \ref{prop:com} yields part $\mathrm{(1)}$ of Theorem \ref{th:decomposition1}.

Clearly, following the proof of Lemma \ref{le:spectrum}, part $\mathrm{(2)}$ of Theorem \ref{th:decomposition1}  holds.

Finally, with the inner product as defined in \eqref{E:inner}, the eigenvalues of $\mathcal{M}_{\partial D}:\overset{\longrightarrow}{\mathrm{curl}}_{\partial D}(H^{\frac{1}{2}}(\partial D))\rightarrow \overset{\longrightarrow}{\mathrm{curl}}_{\partial D}(H^{\frac{1}{2}}(\partial D))$ and the composition $\mathcal{P}_{\overset{\longrightarrow}{\mathrm{curl}}_{\partial D}(H^{\frac32}_{\star}(\partial D))}\circ\mathcal{M}^*_{\partial D}\Big|_{\overset{\longrightarrow}{\mathrm{curl}}_{\partial D}(H^{\frac32}_{\star}(\partial D))}:\overset{\longrightarrow}{\mathrm{curl}}_{\partial D}(H^{\frac{3}{2}}_{\star}(\partial D)) \rightarrow \overset{\longrightarrow}{\mathrm{curl}}_{\partial D}(H^{\frac{3}{2}}_{\star}(\partial D))$ are identical. Note that equality \eqref{E:Product} holds, from which part $\mathrm{(3)}$ of Theorem \ref{th:decomposition1} follows immediately.
\end{proof}

Following arguments analogous to the proofs of Lemmas \ref{le:spectrum}, \ref{le:equivalent}, and Proposition \ref{prop:com}, we establish the following result.
\begin{lemm} \label{le:extra}
The following properties hold for the adjoint operator \(\mathcal{M}^*_{\partial D}\) on the gradient-trace space:

\begin{enumerate}
\item[\rm{(1)}] One has $\sigma(\mathcal{M}^*_{\partial D};\nabla_{\partial D}(H^{\frac{1}{2}}(\partial D)))=\sigma(-\mathcal{K}^*_{\partial D})\backslash\{-1/2\}$.

\item[\rm{(2)}]  The norm $\|\cdot\|_{\mathcal{Q}^{-1}_{\partial D};\nabla_{\partial D}(H^{\frac{1}{2}}(\partial D))}$ induced by the inner product $\langle\cdot,\cdot\rangle_{\mathcal{Q}^{-1}_{\partial D};\nabla_{\partial D}(H^{\frac{1}{2}}(\partial D))}$ defined in \eqref{EE:inner} is equivalent to the Sobolev norm  $\|\cdot\|_{\nabla_{\partial D}(H^{\frac{1}{2}}(\partial D))}$ given by \eqref{space:gradient}.

\item[\rm{(3)}] The operator $\mathcal{M}^*_{\partial D}$ is compact from $\nabla_{\partial D}(H^{\frac{1}{2}}(\partial D))$ to $\nabla_{\partial D}(H^{\frac{1}{2}}(\partial D))$.
\end{enumerate}
\end{lemm}

Similarly, we define an inner product on $\nabla_{\partial D}(H^{\frac32}_{\star}(\partial D))$ by
\begin{align}\label{EE:product}
\langle\boldsymbol w,\boldsymbol v\rangle_{\mathcal{Q}_{\partial D};\nabla_{\partial D}(H^{\frac32}_{\star}(\partial D))}:=\int_{\partial D}\mathcal{Q}_{\partial D}[\boldsymbol w](\boldsymbol x)\cdot \boldsymbol v(\boldsymbol x)\mathrm{d}s(\boldsymbol{x}),\quad\forall \boldsymbol w,\boldsymbol v\in\nabla_{\partial D}(H^{\frac32}_{\star}(\partial D)).
\end{align}

The remainder is to complete the proof of Theorem \ref{th:decomposition2}.

\begin{proof}[Proof of Theorem \ref{th:decomposition2}]
The argument parallels that of Theorem~\ref{th:decomposition1}.  
Part $\mathrm{(1)}$ follows from the self-adjointness and compactness of $\mathcal{M}^*_{\partial D}$ on  
$\nabla_{\partial D}(H^{\frac12}(\partial D))$ with respect to the inner product  
$\langle\cdot,\cdot\rangle_{\mathcal{Q}^{-1}_{\partial D};
\nabla_{\partial D}(H^{\frac12}(\partial D))}$ (recall Theorem \ref{th:symmetrisation2} $\mathrm{(2)}$ and Lemma~\ref{le:extra} $\mathrm{(3)}$).  
Part $(2)$ is a direct consequence of Lemma~\ref{le:extra}~$\mathrm{(1)}$ and the relation between the eigenfunctions of  
$\mathcal{K}^*_{\partial D}$ and $\mathcal{M}^*_{\partial D}$.  
Part $\mathrm{(3)}$ follows from the definition of the inner product \eqref{EE:product} and the self-adjointness of $\mathcal{M}^*_{\partial D}$.
\end{proof}

\section{Electromagnetic problems and boundary localization in the quantum ergodic sense}\label{sec:5}

In this section, we develop a mathematical framework for weak SPRs in electromagnetic systems using the spectral decomposition theorem for static MNP operators. This consequently establishes the boundary localization of weak SPRs in the quantum ergodic sense. Specifically, we complete the proof of Theorem \ref{th:location1}.

\subsection{Scattering problem}

In this subsection, we formulate the scattering problem under consideration.

Recall the scattering problem of a time-harmonic electromagnetic wave incident on a plasmonic nanoparticle $D$ of the form \eqref{eq:D}. The outward unit normal to $\partial D $ is denoted by $\boldsymbol{\nu}$.

Assume that the background  medium is filled with electric permittivity $\epsilon_e$ and magnetic permeability $\mu_e$ and the nanoparticle $D$ is characterized by electric permittivity $\epsilon_c$ and magnetic permeability $\mu_c$.  The corresponding material parameters are described by
\begin{align*}
\epsilon_D&=\epsilon_e\chi(\mathbb{R}^3\setminus\overline{D})+\epsilon_c\chi(D),\\
\mu_D&=\mu_e\chi(\mathbb{R}^3\setminus\overline{D})+\mu_c\chi(D),
\end{align*}
and the corresponding wave number is
\begin{align*}
k_D=\omega\sqrt{\epsilon_D\mu_D}.
\end{align*}
The scattering problem is formulated as follows:
\begin{align}\label{eq:maxwell}
\left\{ \begin{aligned}
&\nabla\times \boldsymbol{E}-\mathrm{i}\omega\mu_e \boldsymbol{H}=0  &&\text{in }\mathbb{R}^3\setminus \overline{D},\\
&\nabla\times \boldsymbol{H}+\mathrm{i}\omega\epsilon_e \boldsymbol{E}=0&&\text{in }\mathbb{R}^3\setminus \overline{D},\\
&\nabla\times \boldsymbol{E}-\mathrm{i}\omega\mu_c \boldsymbol{H}=0  &&\text{in } D,\\
&\nabla\times \boldsymbol{H}+\mathrm{i}\omega\epsilon_c \boldsymbol{E}=0&&\text{in }D
\end{aligned}\right.
\end{align}
subject to the boundary conditions
\begin{align}\label{eq:boundary}
\left\{ \begin{aligned}
& \boldsymbol{\nu}\times \boldsymbol{E}|_{+}= \boldsymbol{\nu}\times \boldsymbol{E}|_{-} &&\text{on }\partial D,\\
& \boldsymbol{\nu}\times \boldsymbol{H}|_{+}= \boldsymbol{\nu}\times \boldsymbol{H}|_{-} &&\text{on }\partial D,
\end{aligned}\right.
\end{align}
and the Silver-M\"{u}ller radiation condition
\begin{align}\label{eq:radiation}
\lim_{| \boldsymbol{x}|\rightarrow+\infty}| \boldsymbol{x}|\left(\sqrt{\mu_e}(\boldsymbol{H}-\boldsymbol{H}^{i})( \boldsymbol x)\times\frac{\boldsymbol{x}}{| \boldsymbol{x}|}-\sqrt{\epsilon_e}(\boldsymbol{E}-\boldsymbol{E}^{i})( \boldsymbol{x})\right)=0,
\end{align}
which holds uniformly in $\boldsymbol{x}/| \boldsymbol{x}|$. 
Here the subscripts $\pm$ represent the limits from the exterior $(+)$ and interior $(-)$ of $D$, respectively.

\subsection{Equivalent boundary integral equations and weak SPRs}

The objective of this subsection is threefold: $\mathrm{(i)}$ to derive the equivalent boundary integral equation formulation for the Maxwell system \eqref{eq:maxwell}, $\mathrm{(ii)}$ to provide rigorous definitions for weak SPRs, and $\mathrm{(iii)}$ to conduct a detailed spectral analysis of weak SPRs.

We first introduce the relevant boundary integral operators. Let $G(k;\boldsymbol{x},\boldsymbol{y})$ stand for the fundamental solution to the Helmholtz operator $\Delta+k^2$ in $\mathbb{R}^3$, given by
\begin{align}\label{Gk:sca-Fun}
G(k;\boldsymbol{x},\boldsymbol{y})=-
\frac{e^{\mathrm{i}k|\boldsymbol{x}-\boldsymbol{y}|}}{4\pi|\boldsymbol{x}-\boldsymbol{y}|}, \quad \boldsymbol{x},\boldsymbol{y}\in\mathbb{R}^3~\text{with}~\boldsymbol{x}\neq \boldsymbol{y}.
\end{align}
By replacing the scalar function $G(\boldsymbol{x},\boldsymbol{y})$ in \eqref{G:sca-Fun}) with  $G(k;\boldsymbol{x},\boldsymbol{y})$, we correspondingly modify the operators in \eqref{OP:M}, \eqref{ssp}, and \eqref{OP:VA} to define
\begin{align*}
\mathcal{M}^k_{\partial D}:\mathbf{H}^{-\frac{1}{2}}_t(\mathrm{div},\partial D)&\rightarrow \mathbf{H}^{-\frac{1}{2}}_t(\mathrm{div},\partial D)\nonumber\\
\boldsymbol{\phi}&\mapsto\mathcal{M}^k_{\partial D}[\boldsymbol{\phi}](\boldsymbol{x})
=\int_{\partial D}\boldsymbol{\nu}_{\boldsymbol{x}}\times \nabla_{\boldsymbol{x}}\times  G(k;\boldsymbol{x},\boldsymbol{y})\boldsymbol{\phi}(\boldsymbol{y})\mathrm{d}s(\boldsymbol{y});\\
\mathcal{S}^{k}_{\partial D}:H^{-\frac{1}{2}}(\partial D)&\rightarrow H^{\frac{1}{2}}(\partial D)\text{ or } H^1_{\mathrm{loc}}(\mathbb{R}^3)\\
\varphi& \mapsto\mathcal{S}^{k}_{\partial D}[\phi](\boldsymbol{x})=\int_{\partial D}  G(k;\boldsymbol{x},\boldsymbol{y})\phi(\boldsymbol{y})\mathrm{d}s(\boldsymbol{y}); \\
\vec{\mathcal{S}}^k_{\partial D}:{H}^{-\frac{1}{2}}(\partial D)^3&\rightarrow {H}^{\frac{1}{2}}(\partial D)^3\text{ or } H^1_{\mathrm{loc}}(\mathbb{R}^3)^3\nonumber\\
\boldsymbol{\xi}&\mapsto\vec{\mathcal{S}}^k_{\partial D}[\boldsymbol{\xi}](\boldsymbol{x})=\int_{\partial D} G(k;\boldsymbol{x},\boldsymbol{y})\boldsymbol{\xi}(\boldsymbol{y})\mathrm{d}s(\boldsymbol{y}).
\end{align*}
In addition, we introduce the operator
\begin{align*}
\mathcal{L}^k_{\partial D}:\mathbf{H}^{-\frac{1}{2}}_t(\mathrm{div},\partial D)&\rightarrow \mathbf{H}^{-\frac{1}{2}}_t(\mathrm{div},\partial D)\nonumber\\
\boldsymbol{\phi}&\mapsto\mathcal{L}^k_{\partial D}[\boldsymbol{\phi}](\boldsymbol{x})
=\boldsymbol{\nu}_{\boldsymbol{x}}\times( k^2\vec{\mathcal{S}}^k_{\partial D}[\boldsymbol{\phi}](\boldsymbol{x})+\nabla\mathcal{S}^k_{\partial D}[\nabla_{\partial D}\cdot \boldsymbol{\phi}](\boldsymbol{x})).
\end{align*}

The solution to equation \eqref{eq:maxwell} (see also \cite{Ammari2016Mathematical,Ammari2016Surface}) can be presented by
\begin{align}\label{EX:E}
\boldsymbol{E}(\boldsymbol{x})=
\left\{ \begin{aligned}
&\boldsymbol{E}^i(\boldsymbol x)+\mu_e\nabla\times\vec{\mathcal{S}}^{k_e}_{\partial D}[\boldsymbol{\psi}](\boldsymbol x)+\nabla\times\nabla\times\vec{\mathcal{S}}^{k_e}_{\partial D}[\boldsymbol{\varphi}](\boldsymbol x), && \boldsymbol{x}\in\mathbb{R}^3\backslash\overline{D},\\
&\mu_c\nabla\times\vec{\mathcal{S}}^{k_c}_{\partial D}[\boldsymbol{\psi}](\boldsymbol x)+\nabla\times\nabla\times\vec{\mathcal{S}}^{k_c}_{\partial D}[\boldsymbol{\varphi}](\boldsymbol x), &&\boldsymbol{x}\in D,
\end{aligned}\right.
\end{align}
and
\begin{align}\label{EX:H}
\boldsymbol H(x)=
\left\{ \begin{aligned}
&\boldsymbol{H}^i(\boldsymbol x)-\frac{\mathrm{i}}{\omega}\nabla\times\nabla\times\vec{\mathcal{S}}^{k_e}_{\partial D}[\boldsymbol{\psi}](\boldsymbol x)-\frac{\mathrm{i}}{\omega\mu_e}k^2_e\nabla\times\vec{\mathcal{S}}^{k_e}_{\partial D}[\boldsymbol{\varphi}](\boldsymbol x), &&\boldsymbol{x}\in\mathbb{R}^3\backslash\overline{D},\\
&-\frac{\mathrm{i}}{\omega}\nabla\times\nabla\times\vec{\mathcal{S}}^{k_c}_{\partial D}[\boldsymbol{\psi}](\boldsymbol x)-\frac{\mathrm{i}}{\omega\mu_c}k^2_c\nabla\times\vec{\mathcal{S}}^{k_c}_{\partial D}[\boldsymbol{\varphi}](\boldsymbol x), &&\boldsymbol{x}\in D.
\end{aligned}\right.
\end{align}
Using \eqref{SS:jump}, \eqref{EX:E}, \eqref{EX:H}, and two transmission conditions in \eqref{eq:boundary}, we obtain the boundary integral system
\begin{align}\label{EQ:equation}
\mathscr{A}_{\partial D}(\omega)
\begin{pmatrix}
\boldsymbol{\psi}\\
\boldsymbol{\varphi}
\end{pmatrix}(\boldsymbol{x})
=
\begin{pmatrix}
\boldsymbol{\nu}_{\boldsymbol{x}}\times \boldsymbol{E}^i\\
\mathrm{i}\omega\boldsymbol{\nu}_{\boldsymbol{x}}\times \boldsymbol{H}^i
\end{pmatrix}
(\boldsymbol{x}),
\end{align}
where
\begin{align*}
\mathscr{A}_{\partial D}(\omega):=
\begin{pmatrix}
\dfrac{\mu_c+\mu_e}{2}\mathcal{I}+\mu_c\mathcal{M}^{k_c}_{\partial D}-\mu_e\mathcal{M}^{k_e}_{\partial D}&\mathcal{L}^{k_c}_{\partial D}-\mathcal{L}^{k_e}_{\partial D}\\
\mathcal{L}^{k_c}_{\partial D}-\mathcal{L}^{k_e}_{\partial D}&\dfrac{1}{2}(\dfrac{k^2_c}{\mu_c}+\dfrac{k^2_e}{\mu_e})\mathcal{I}+(\dfrac{k^2_c}{\mu_c}\mathcal{M}^{k_c}_{\partial D}-\dfrac{k^2_e}{\mu_e}\mathcal{M}^{k_e}_{\partial D})
\end{pmatrix}.
\end{align*}
System \eqref{EQ:equation} is an equivalent boundary-integral formulation of the Maxwell scattering problem \eqref{eq:maxwell}.

Next, let us introduce the scaling transformation $\boldsymbol x=\boldsymbol z+\delta \tilde{\boldsymbol x}$. For  the function $\boldsymbol{\phi}$ on $\partial D$, we define the corresponding function $\tilde{\boldsymbol\phi}$ on $\partial B$ by
\begin{align*}
\tilde{\boldsymbol\phi}(\tilde{\boldsymbol x}):=\boldsymbol\phi(\boldsymbol z+\delta \tilde{\boldsymbol x}).
\end{align*}
It is straightforward that
\begin{align}\label{G:multi}
G(k;\boldsymbol{x},\boldsymbol{y})\stackrel{\eqref{Gk:sca-Fun}}{=}\delta^{-1}G(\delta k;\tilde{\boldsymbol{x}},\tilde{\boldsymbol{y}})=\delta^{-1}\left(-\frac{1}{4\pi|\tilde{\boldsymbol{x}}-\tilde{\boldsymbol{y}}|}-\sum^{+\infty}_{j=1}\delta^j\frac{(\mathrm{i}k)^j}{4\pi j!}|\tilde{\boldsymbol{x}}-\tilde{\boldsymbol{y}}|^{j-1}\right).
\end{align}
A direct consequence of \eqref{G:multi} is the following asymptotic expansions of the
boundary integral operators.
\begin{lemm}\label{le:M}
For $\boldsymbol\phi\in \mathbf{H}^{-\frac{1}{2}}_t(\mathrm{div},\partial D)$, the following asymptotic expansion holds as $\delta\rightarrow 0^+$:
\begin{align*}
&\mathcal{M}^{k}_{\partial D}[\boldsymbol\phi](\boldsymbol x)=
\mathcal{M}_{\partial B}[\tilde{\boldsymbol\phi}](\tilde{\boldsymbol x})+\sum^{+\infty}_{j=2}\delta^j\mathcal{M}^{k}_{\partial B,j}[\tilde{\boldsymbol\phi}](\tilde{\boldsymbol x}),\\
&\mathcal{L}^{k_c}_{\partial D}[\boldsymbol\phi](\boldsymbol x)-\mathcal{L}^{k_e}_{\partial D}[\boldsymbol\phi](\boldsymbol x)=\sum^{+\infty}_{j=1}\delta^j\omega\mathcal{L}_{\partial B,j}[\tilde{\boldsymbol\phi}](\tilde{\boldsymbol x}),
\end{align*}
where
\begin{align*}
\mathcal{M}^{k}_{\partial B,j}[\tilde{\boldsymbol\phi}](\tilde{\boldsymbol x})&=\int_{\partial B}\frac{-(\mathrm{i}k)^j}{4\pi j!}\boldsymbol{\nu}_{\tilde{\boldsymbol x}}\times\nabla_{\tilde{\boldsymbol{x}}}\times|\tilde{\boldsymbol{x}}-\tilde{\boldsymbol{y}}|^{j-1}\tilde{\boldsymbol\phi}(\tilde{\boldsymbol x})\mathrm{d}s(\tilde{\boldsymbol y}),\\
\mathcal{L}_{\partial B,j}[\tilde{\boldsymbol{\phi}}](\tilde{\boldsymbol x})&=C_j\boldsymbol{\nu}_{\tilde{\boldsymbol x}}\times(\int_{\partial B}|\tilde{\boldsymbol{x}}-\tilde{\boldsymbol{y}}|^{j-2}\tilde{\boldsymbol{\phi}}(\tilde{\boldsymbol y})\mathrm{d}s(\tilde{\boldsymbol y})\\
&\quad-\int_{\partial B}\frac{|\tilde{\boldsymbol{x}}-\tilde{\boldsymbol{y}}|^{j-2}(\tilde{\boldsymbol{x}}-\tilde{\boldsymbol{y}})}{j+1}\nabla_{\partial B}\cdot\tilde{\boldsymbol{\phi}}(\tilde{\boldsymbol y})\mathrm{d}s(\tilde{\boldsymbol y})),
\end{align*}
with  constants $C_j=\dfrac{\mathrm{i}^{j}(k_c^{j+1}-k_e^{j+1})}{4\pi\omega(j-1)!}$. 
\end{lemm}

Similarly, we define the scaled densities and incident fields
\begin{align}
\tilde{\boldsymbol\psi}(\tilde{\boldsymbol x})&:=\boldsymbol\psi(\boldsymbol z+\delta \tilde{\boldsymbol x}),\quad\tilde{\boldsymbol\varphi}(\tilde{\boldsymbol x}):=\boldsymbol\varphi(\boldsymbol z+\delta \tilde{\boldsymbol x}),\nonumber\\
\tilde{\boldsymbol{E}}^i(\tilde{\boldsymbol x})&:=\boldsymbol{E}^i(\boldsymbol z+\delta \tilde{\boldsymbol x}),\quad\tilde{\boldsymbol{H}}^i(\tilde{\boldsymbol x}):=\boldsymbol{H}^i(\boldsymbol z+\delta \tilde{\boldsymbol x}).\label{ES:incident}
\end{align}
According to Lemma \ref{le:M}, we can rewrite the boundary integral system \eqref{EQ:equation} in the scaled variables as
\begin{align}\label{EQB:equation}
\mathscr{A}_{\partial B}(\delta)
\begin{pmatrix}
\tilde{\boldsymbol{\psi}}\\
\omega\tilde{\boldsymbol{\varphi}}
\end{pmatrix}(\tilde{\boldsymbol{x}})
=
\begin{pmatrix}
\dfrac{\boldsymbol{\nu}_{\tilde{\boldsymbol{x}}}\times \tilde{\boldsymbol{E}}^i}{\mu_e-\mu_c}\\
\dfrac{\mathrm{i}\boldsymbol{\nu}_{\tilde{\boldsymbol{x}}}\times \tilde{\boldsymbol{H}}^i}{\epsilon_e-\epsilon_c}
\end{pmatrix}
(\tilde{\boldsymbol{x}}),
\end{align}
where $\mathscr{A}_{\partial B}(\delta)=\mathscr{A}_{\partial B}(0)+\mathscr{B}_{\partial B}(\delta)$ with
\begin{align*}
\mathscr{A}_{\partial B}(0)&=
\begin{pmatrix}
\dfrac{\mu_c+\mu_e}{2(\mu_e-\mu_c)}\mathcal{I}-\mathcal{M}_{\partial B}&0\\
0&\dfrac{\epsilon_c+\epsilon_e}{2(\epsilon_e-\epsilon_c)}\mathcal{I}-\mathcal{M}_{\partial B}
\end{pmatrix}
,\\
\mathscr{B}_{\partial B}(\delta)&=
\begin{pmatrix}
\delta^2\dfrac{\mu_c\mathcal{M}^{k_c}_{\partial B,2}-\mu_e\mathcal{M}^{k_e}_{\partial B,2}}{\mu_e-\mu_c}+\mathcal{O}(\delta^3)&\dfrac{1}{\mu_e-\mu_c}(\delta \mathcal{L}_{\partial B,1}+\delta^2\mathcal{L}_{\partial B,2}+\mathcal{O}(\delta^3))\\
\dfrac{1}{\epsilon_e-\epsilon_c}(\delta \mathcal{L}_{\partial B,1}+\delta^2\mathcal{L}_{\partial B,2}+\mathcal{O}(\delta^3))&\delta^2\dfrac{\epsilon_c\mathcal{M}^{k_c}_{\partial B,2}-\epsilon_e\mathcal{M}^{k_e}_{\partial B,2}}{\epsilon_e-\epsilon_c}+\mathcal{O}(\delta^3)
\end{pmatrix}
.\nonumber
\end{align*}

In analogy with the definitions of weak Minnaert resonances by \cite{Li2022Minnaert}, we introduce  the corresponding notions of  weak SPRs for nanoparticles. 


\begin{defi}\label{def:weak}

For  fixed $\omega > 0$, if there exists a nontrivial function
\begin{align*}
\begin{pmatrix} \check{\boldsymbol{\psi}} \\ \omega\check{\boldsymbol{\varphi}} \end{pmatrix}=
\begin{pmatrix} \check{\boldsymbol{\psi}}(\delta) \\ \omega\check{\boldsymbol{\varphi}}(\delta) \end{pmatrix} \in \mathbf{H}^{-\frac{1}{2}}_t(\mathrm{div},\partial B)^2,
\end{align*}
with $
C^{-1}\leq\left\| \begin{pmatrix} \check{\boldsymbol{\psi}} \\ \omega\check{\boldsymbol{\varphi}} \end{pmatrix} \right\|_{\mathbf{H}^{-\frac{1}{2}}_t(\mathrm{div},\partial B)^2}\leq C$ for some constant $C>0$, such that,  as $\delta \rightarrow 0^+$,
\begin{align}\label{eq:resonance2}
\left\| \mathscr{A}_{\partial B}(\delta) \begin{pmatrix} \check{\boldsymbol{\psi}} \\ \omega\check{\boldsymbol{\varphi}} \end{pmatrix} \right\|_{\mathbf{H}^{-\frac{1}{2}}_t(\mathrm{div},\partial B)^2}\to 0,
\end{align}
then we say that a weak SPR of system \eqref{EQB:equation} occurs.
\end{defi}

According to Definition \ref{def:weak}, the occurrence of a weak SPR in system \eqref{EQB:equation} is closely related to the nontrivial leading-order term of $\begin{pmatrix} \check{\boldsymbol{\psi}} \\ \omega\check{\boldsymbol{\varphi}} \end{pmatrix}$, given by
\begin{align}\label{eq:condition}
\begin{pmatrix} 
\tilde{\boldsymbol{\psi}} \\
\omega\tilde{\boldsymbol{\varphi}} 
\end{pmatrix} \in \overset{\longrightarrow}{\mathrm{curl}}_{\partial B}(H^{\frac{1}{2}}(\partial B))^2\subset\mathbf{H}^{-\frac{1}{2}}_t(\mathrm{div},\partial B)^2,
\end{align}
with $C^{-1}\leq\left\|\begin{pmatrix} 
\tilde{\boldsymbol{\psi}} \\
\omega\tilde{\boldsymbol{\varphi}} 
\end{pmatrix}\right\|_{\overset{\longrightarrow}{\mathrm{curl}}_{\partial B}(H^{\frac{1}{2}}(\partial B))^2}\leq C$. For fixed $\omega>0$, this leading-order term satisfies
\begin{align}\label{OP2:B1}
\mathscr{A}_{\partial B}(0)
\begin{pmatrix}
\tilde{\boldsymbol{\psi}} \\
\omega\tilde{\boldsymbol{\varphi}}
\end{pmatrix}
(\tilde{\boldsymbol{x}}) = 0.
\end{align}
In other words, if such a nontrivial solution to \eqref{OP2:B1} satisfying \eqref{eq:condition} exists, then it induces a weak SPR of \eqref{EQB:equation}.

\subsection{Assumptions}

We first explain why the nontrivial solution of system \eqref{OP2:B1} satisfies condition \eqref{eq:condition}. Afterwards, we will introduce some additional assumptions for the convenience of the study.

In our analysis, we consider the incident field configuration under consideration, specifically the excitation by a point source.

Generally, let $
\begin{pmatrix}
\boldsymbol{E}^i\\
\boldsymbol{H}^i
\end{pmatrix}
$
 be the incident electromagnetic field which is a solution to Maxwell's equations
\begin{align*}
\left\{ \begin{aligned}
&\nabla\times \boldsymbol{E}^i-\mathrm{i}\omega\mu_e \boldsymbol{H}^i=0  &&\text{in }\mathbb{R}^3,\\
&\nabla\times \boldsymbol{H}^i+\mathrm{i}\omega\epsilon_e \boldsymbol{E}^i=-\mathrm{i}\frac{1}{\omega\mu_e}\boldsymbol{p}\boldsymbol{\delta}_{\boldsymbol{s}} &&\text{in }\mathbb{R}^3,
\end{aligned}\right.
\end{align*}
where $\boldsymbol{\delta}_{\boldsymbol{s}}(\boldsymbol{x})$ denotes  the delta function, and  $\boldsymbol{s}=(s_l)_{l=1}^3\in\mathbb{R}^3$ is the source location with a dipole moment $\boldsymbol{p}=(p_l)_{l=1}^3 \in\mathbb{R}^3$. Using the scalar function $G(k;\cdot,\cdot)$ from \eqref{Gk:sca-Fun}, we define the matrix-valued function 
\begin{align}\label{G:vec-Fun}
\boldsymbol{\mathcal{G}}(k;\boldsymbol{x},\boldsymbol{y}):=-\frac{\delta^3}{k^2}\nabla_{\boldsymbol{x}}\nabla_{\boldsymbol{x}}\cdot(G(k;\boldsymbol{x},\boldsymbol{y})\boldsymbol{\mathfrak{I}})-\delta^3 G(k;\boldsymbol{x},\boldsymbol{y})\boldsymbol{\mathfrak{I}},
\end{align}
where $\boldsymbol{\mathfrak{I}}$ is the $3\times 3$ identity matrix. Then the incident electromagnetic field  is given by
\begin{align}\label{EH:incident}
\begin{pmatrix}
\boldsymbol{E}^{i}(\boldsymbol{x})\\
\boldsymbol{H}^i(\boldsymbol{x})
\end{pmatrix}
=
\begin{pmatrix}
\boldsymbol{\mathcal{G}}(k_e;\boldsymbol{x},\boldsymbol{s})\boldsymbol{p}\\
-\dfrac{\mathrm{i}}{\omega\mu_e}\nabla\times(\boldsymbol{\mathcal{G}}(k_e;\boldsymbol{x},\boldsymbol{s})\boldsymbol{p}).
\end{pmatrix}
\end{align}
Applying the scaling $\boldsymbol{x}=\boldsymbol{z}+\delta\tilde{\boldsymbol{x}}$ and using
\eqref{ES:incident}, \eqref{G:vec-Fun}, \eqref{EH:incident}, and \eqref{G:multi}, we obtain the scaled incident field
\begin{align}\label{TE:incident}
\begin{pmatrix}
\tilde{\boldsymbol{E}}^i(\tilde{\boldsymbol x})\\
\tilde{\boldsymbol{H}}^i(\tilde{\boldsymbol x})
\end{pmatrix}
=
\begin{pmatrix}
-\dfrac{1}{k^2_e}\nabla_{\tilde{\boldsymbol x}}\nabla_{\tilde{\boldsymbol x}}\cdot(G(\delta k_e;\tilde{\boldsymbol{x}},\tilde{\boldsymbol{s}})\boldsymbol{p})-\delta^2G(\delta k_e;\tilde{\boldsymbol{x}},\tilde{\boldsymbol{s}})\boldsymbol{p}\\
\dfrac{\mathrm{i}\delta}{\omega\mu_e}\nabla_{\tilde{\boldsymbol x}}\times(G(\delta k_e;\tilde{\boldsymbol{x}},\tilde{\boldsymbol{s}})\boldsymbol{p})
\end{pmatrix},
\end{align}
where $\tilde{\boldsymbol{s}} = (\boldsymbol{s}-\boldsymbol{z})/\delta$.

Let $B_R \subset \mathbb R^3$ be a ball centered at the origin with radius $R>0$ such that  
$\overline{B}\subset B_R$ and $\tilde{\boldsymbol{s}}\notin B_R$. 
According to \cite{Buffa2002Traces},  the tangential trace operator maps $\mathbf{H}(\mathrm{curl},B_{R}\backslash\overline{B})$ continuously into $\mathbf{H}^{-\frac{1}{2}}_t(\mathrm{div},\partial B)$. Since 
\begin{align*}
\begin{pmatrix}
\tilde{\boldsymbol{E}}^i\\
\tilde{\boldsymbol{H}}^i
\end{pmatrix}\in \mathbf{H}(\mathrm{curl},B_{R}\backslash\overline{B})^2,
\end{align*}
it follows that
\begin{align*}
\begin{pmatrix}
\boldsymbol{\nu}\times\tilde{\boldsymbol{E}}^i\\
\boldsymbol{\nu}\times\tilde{\boldsymbol{H}}^i
\end{pmatrix}
\in\mathbf{H}^{-\frac{1}{2}}_t(\mathrm{div},\partial B)^2.
\end{align*}
Moreover, using \eqref{TE:incident},  we obtain the asymptotic expansions
\begin{align*}
\begin{aligned}
\boldsymbol{\nu}_{\tilde{\boldsymbol x}}\times\tilde{\boldsymbol{E}}^i(\tilde{\boldsymbol x})&=\frac{1}{k^2_e}\overset{\longrightarrow}{\mathrm{curl}}_{\partial B}\nabla_{\tilde{\boldsymbol{x}}}\cdot (G(\tilde{\boldsymbol{x}},\tilde{\boldsymbol{s}})\boldsymbol{p})+\mathcal{O}(\delta^2) &&\text{on}~~\partial B,\\
\boldsymbol{\nu}_{\tilde{\boldsymbol x}}\times\tilde{\boldsymbol{H}}^i(\tilde{\boldsymbol x})&=\mathcal{O}(\delta)&&\text{on}~~\partial B.
\end{aligned}
\end{align*}
Hence, it can be from the equality $\nabla_{\partial B}\cdot\frac{1}{k^2_e}\overset{\longrightarrow}{\mathrm{curl}}_{\partial B}\nabla_{\tilde{\boldsymbol x}}\cdot(G(\tilde{\boldsymbol{x}},\tilde{\boldsymbol{s}})\boldsymbol{p})=0$ obtained that the leading term
\begin{align*}
\frac{1}{k^2_e}\overset{\longrightarrow}{\mathrm{curl}}_{\partial B}\nabla\cdot(G(\cdot,\tilde{\boldsymbol{s}})\boldsymbol{p})\in\overset{\longrightarrow}{\mathrm{curl}}_{\partial B}(H^{\frac{1}{2}}(\partial B))^2,
\end{align*}
where we have applied \cite[Corollary 5.3]{Buffa2002Traces}. This means that the principal part of the tangential trace associated with  the incident field 
$
\begin{pmatrix}
\boldsymbol{\nu}\times\tilde{\boldsymbol{E}}^i\\
\boldsymbol{\nu}\times\tilde{\boldsymbol{H}}^i\end{pmatrix}
$  belongs to $\overset{\longrightarrow}{\mathrm{curl}}_{\partial B}(H^{\frac{1}{2}}(\partial B))^2$.  Therefore, for the well-posed system \eqref{eq:maxwell}, subject to \eqref{eq:boundary} and \eqref{eq:radiation}, the leading-order term of the density function $\begin{pmatrix} 
\tilde{\boldsymbol{\psi}} \\
\omega\tilde{\boldsymbol{\varphi}} 
\end{pmatrix} \in \overset{\longrightarrow}{\mathrm{curl}}_{\partial B}(H^{\frac{1}{2}}(\partial B))^2$.
Based on the above  observation, system \eqref{OP2:B1} admits a nontrivial solution $\begin{pmatrix} \tilde{\boldsymbol{\psi}} \\ \omega\tilde{\boldsymbol{\varphi}} \end{pmatrix}$ satisfying condition \eqref{eq:condition}.

In order to analyze and resolve this problem, we need  additional assumption.

\begin{assu}\label{assumption2}
Without loss of generality, we set
\begin{align*}
\epsilon_e=\mu_e=1,
\end{align*}
and for each $j\in\mathbb{Z}_+$, we take
\begin{align*}
\epsilon_c=\mu_c=-\tau_j,
\end{align*}
where $\tau_j>0$ and $\tau_j\neq1$.
\end{assu}
Under Assumption \ref{assumption2}, for fixed $\omega\in\mathbb{R}_+$, our object is to seek a nontrivial function $\begin{pmatrix}
\tilde{\boldsymbol{\psi}}\\
\omega\tilde{\boldsymbol{\varphi}}
\end{pmatrix}$, with condition \eqref{eq:condition}, 
is the solution of the following system
\begin{align}\label{OP2:B}
\tilde{\mathscr{A}}_{\partial B,j}(0)
\begin{pmatrix}
\tilde{\boldsymbol{\psi}}\\
\omega\tilde{\boldsymbol{\varphi}}
\end{pmatrix}
(\tilde{\boldsymbol{x}})
=0,
\end{align}
where
\begin{align}\label{E:operator-A}
\tilde{\mathscr{A}}_{\partial B,j}(0):=
\begin{pmatrix}
\frac{1-\tau_j}{2(1+\tau_j)}\mathcal{I}-\mathcal{M}_{\partial B}&0\\
0&\frac{1-\tau_j}{2(1+\tau_j)}\mathcal{I}-\mathcal{M}_{\partial B}
\end{pmatrix}.
\end{align}

\begin{rema}

Let $\mathbb{S}$ denote  the unit sphere, namely, 
\begin{align}\label{eq:sphere}
\mathbb{S}=\{\boldsymbol x\in\mathbb{R}^3,|\boldsymbol x|=1\}.
\end{align}
Under the condition $\epsilon_e=\mu_e=1$ specified in Assumption \ref{assumption2}, for any tangential vector field $\boldsymbol{h}\in L^{2}(\mathbb{S})^3$, the vector field $
\begin{pmatrix}
\tilde{\boldsymbol{E}}_{\boldsymbol{h}}\\
\tilde{\boldsymbol{H}}_{\boldsymbol{h}}
\end{pmatrix},
$ given by
\begin{align*}
\tilde{\boldsymbol{E}}_{\boldsymbol{h}}(\tilde{\boldsymbol{x}}):=\int_{\mathbb{S}}e^{\mathrm{i}\omega\tilde{\boldsymbol{x}}\cdot \boldsymbol{\vartheta}}\boldsymbol{h}(\boldsymbol{\vartheta})\mathrm{d}s(\boldsymbol{\vartheta}),
\quad\tilde{\boldsymbol{H}}_{\boldsymbol{h}}(\tilde{\boldsymbol{x}}):=\frac{1}{\omega}\int_{\mathbb{S}}e^{\mathrm{i}\omega\tilde{\boldsymbol{x}}\cdot \boldsymbol{\vartheta}}\boldsymbol{\vartheta}\times\boldsymbol{h}(\boldsymbol{\vartheta})\mathrm{d}s(\boldsymbol{\vartheta}),
\end{align*}
is called the ``Maxwell--Herglotz field'' associated with  $\boldsymbol{h}$; see \cite{Colton2002on,Weck2004Approximation}. As given in \cite[Theorem 3]{Weck2004Approximation}, when $B$ is of class $C^1$, if $\omega$ is not an eigenvalue of Maxwell's boundary value problem subject to the boundary condition of $\boldsymbol{\nu}\times \tilde{\boldsymbol{E}}=0$, then the set 
\begin{align*}
\left\{\boldsymbol{\nu}\times \tilde{\boldsymbol{E}}:
\begin{pmatrix}
\tilde{\boldsymbol{E}}\\
\tilde{\boldsymbol{H}}
\end{pmatrix}
\text{ is a Maxwell--Herglotz field}\right\}
\end{align*}
is dense in the space of all tangential vector fields belonging to $H^{-\frac{1}{2}}(\partial B)^3$.

Under the above framework, we can always approximately construct the incident electromagnetic field
$
\begin{pmatrix}\tilde{\boldsymbol{E}}^{i}\\\tilde{\boldsymbol{H}}^{i}\end{pmatrix}
$
in $B_{R}\setminus \overline{B}$. This field satisfies the time-harmonic Maxwell equations
\begin{align*}
\left\{ \begin{aligned}
&\nabla\times \tilde{\boldsymbol{E}}^i-\mathrm{i}\omega\tilde{\boldsymbol{H}}^i=0  &&\text{in }B_{R}\setminus \overline{B},\\
&\nabla\times \tilde{\boldsymbol{H}}^i+\mathrm{i}\omega \tilde{\boldsymbol{E}}^i=0 &&\text{in }B_{R}\setminus \overline{B},
\end{aligned}\right.
\end{align*}
with prescribed tangential traces
\begin{align*}
\boldsymbol{\nu}\times \tilde{\boldsymbol{E}}^i=\tilde{\boldsymbol{h}}_1,\quad\boldsymbol{\nu}\times \tilde{\boldsymbol{H}}^i=\tilde{\boldsymbol{h}}_2,
\end{align*}
where the pair $
\begin{pmatrix}\tilde{\boldsymbol{h}}_1\\
\tilde{\boldsymbol{h}}_2
\end{pmatrix}
\in\overset{\longrightarrow}{\mathrm{curl}}_{\partial B}(H^{\frac{1}{2}}(\partial B))^2$.
\end{rema}

\subsection{boundary localization on bounded nanoparticles}

In this subsection, we continue to seek the nontrivial solution to equation \eqref{OP2:B} in  an appropriate function space.  We recast this problem as a spectral problem for the static MNP operator introduced  in \eqref{OP:M}. In addition, we quantitatively investigate the boundary localization or equivalently the
decay off the boundary surface of the plasmon when the surface $\partial D$
 is of arbitrary
shape. Our objective is to establish the boundary localization for the sequence of weak plasmons.

\begin{prop}
Under Assumption \ref{assumption2}, if  the Maxwell system  \eqref{eq:maxwell} exhibits weak SPRs for sufficiently small $\delta>0$, then 
for every $j\in \mathbb{Z}_+$, the following holds:
\begin{enumerate}
\item[\rm{(1)}] System \eqref{OP2:B} admits a family of nontrivial solutions
\begin{align*}
\begin{pmatrix}
\boldsymbol{\phi}_j\\
\omega\boldsymbol{\phi}_j
\end{pmatrix}
\in\overset{\longrightarrow}{\mathrm{curl}}_{\partial D}(H^{\frac{1}{2}}(\partial D))^2
\end{align*}
related to the material parameters $\tau_j$ defined in Assumption \ref{assumption2}, where $\boldsymbol{\phi}_j(\boldsymbol{x}):=\tilde{\boldsymbol{\phi}}_j(\frac{\boldsymbol{x}-\boldsymbol{z}}{\delta})$. Here $\tilde{\boldsymbol{\phi}}_j$ are the normalized eigenfunctions of $\mathcal{M}_{\partial B}$ defined in \eqref{OP:M} with respect to the inner product $\langle\cdot,\cdot\rangle_{\mathcal{N}^{-1}_{\partial B};\overset{\longrightarrow}{\mathrm{curl}}_{\partial B}(H^{\frac{1}{2}}(\partial B))}$  introduced in \eqref{E:inner}.

\item[\rm{(2)}] One obtains the following weak plasmon sequence $\{(\boldsymbol{E}_j,\boldsymbol{H}_j)\}_{j\in\mathbb{Z}_+}$, where
\begin{align}\label{E:mode}
\boldsymbol{E}_j(\boldsymbol{x})=
\left\{ \begin{aligned}
&\mu_e\nabla\times\vec{\mathcal{S}}^{k_e}_{\partial D}[\boldsymbol{\phi}_j](\boldsymbol x)+\nabla\times\nabla\times\vec{\mathcal{S}}^{k_e}_{\partial D}[\boldsymbol{\phi}_j](\boldsymbol x), && \boldsymbol{x}\in\mathbb{R}^3\backslash\overline{D},\\
&\mu_c\nabla\times\vec{\mathcal{S}}^{k_c}_{\partial D}[\boldsymbol{\phi}_j](\boldsymbol x)+\nabla\times\nabla\times\vec{\mathcal{S}}^{k_c}_{\partial D}[\boldsymbol{\phi}_j](\boldsymbol x), &&\boldsymbol{x}\in D,
\end{aligned}\right.
\end{align}
and
\begin{align}\label{H:mode}
\boldsymbol H_j(x)=
\left\{ \begin{aligned}
&-\frac{\mathrm{i}}{\omega}\nabla\times\nabla\times\vec{\mathcal{S}}^{k_e}_{\partial D}[\boldsymbol{\phi_j}](\boldsymbol x)-\frac{\mathrm{i}}{\omega\mu_e}k^2_e\nabla\times\vec{\mathcal{S}}^{k_e}_{\partial D}[\boldsymbol{\phi}_j](\boldsymbol x), &&\boldsymbol{x}\in\mathbb{R}^3\backslash\overline{D},\\
&-\frac{\mathrm{i}}{\omega}\nabla\times\nabla\times\vec{\mathcal{S}}^{k_c}_{\partial D}[\boldsymbol{\phi}_j](\boldsymbol x)-\frac{\mathrm{i}}{\omega\mu_c}k^2_c\nabla\times\vec{\mathcal{S}}^{k_c}_{\partial D}[\boldsymbol{\phi}_j](\boldsymbol x), &&\boldsymbol{x}\in D.
\end{aligned}\right.
\end{align}
\end{enumerate}
\end{prop}
\begin{proof}
Let $
\begin{pmatrix}
\tilde{\boldsymbol{\psi}}\\
\omega\tilde{\boldsymbol{\varphi}}
\end{pmatrix}\in \overset{\longrightarrow}{\mathrm{curl}}_{\partial B}(H^{\frac{1}{2}}(\partial B))^2
$. With respect to the inner product $\langle\cdot,\cdot\rangle_{\mathcal{N}^{-1}_{\partial B};\overset{\longrightarrow}{\mathrm{curl}}_{\partial B}(H^{\frac{1}{2}}(\partial B))}$, we define the normalized eigenfunctions $\tilde{\boldsymbol{\phi}}_j$ of $\mathcal{M}_{\partial B}$ by
\begin{align*}
\tilde{\boldsymbol{\phi}}_j(\tilde{\boldsymbol{x}}):=\delta^{\frac{3}{2}}\boldsymbol{\varphi}_j(\boldsymbol{z}+\delta\tilde{\boldsymbol{x}}),
\end{align*}
where $\boldsymbol{\varphi}_j$ are given by \eqref{E:eigenfunction}. Setting 
\begin{align}\label{eq:weak}
\begin{pmatrix}
\tilde{\boldsymbol{\psi}}\\
\omega\tilde{\boldsymbol{\varphi}}
\end{pmatrix}
=
\begin{pmatrix}
\tilde{\boldsymbol{\phi}}_j\\
\omega\tilde{\boldsymbol{\phi}}_j
\end{pmatrix},
\end{align}
we can obtain
\begin{align*}
\begin{pmatrix}
\frac{1-\tau}{2(1+\tau_j)}\mathcal{I}-\mathcal{M}_{\partial B}&0\\
0&\frac{1-\tau}{2(1+\tau_j)}\mathcal{I}-\mathcal{M}_{\partial B}
\end{pmatrix}
\begin{pmatrix}
\tilde{\boldsymbol{\phi}}_j\\
\omega\tilde{\boldsymbol{\phi}}_j
\end{pmatrix}=\left(\frac{1-\tau_j}{2(1+\tau_j)}-\lambda_j \right)
\begin{pmatrix}
\tilde{\boldsymbol{\phi}}_j\\
\omega\tilde{\boldsymbol{\phi}}_j
\end{pmatrix}.
\end{align*}
Clearly, the eigenvalue $\lambda_j$ of $\mathcal{M}_{\partial B}$ satisfies
\begin{align*}
\lambda_j = \frac{1-\tau_j}{2(1+\tau_j)}.
\end{align*}
Then  the weak SPRs are induced by condition \eqref{eq:resonance2}, i.e., property $\mathrm{(1)}$ holds.

Furthermore, combining formula \eqref{eq:weak} with \eqref{EX:E} and \eqref{EX:H}, we obtain the weak plasmon sequence $\{(\boldsymbol{E}_j,\boldsymbol{H}_j)\}_{j\in\mathbb{Z}_+}$, which satisfies both \eqref{E:mode} and \eqref{H:mode}.

The proof is complete.
\end{proof}

We begin by recalling a key result from \cite[Lemma 2.1]{Ando2021surface} concerning the 
almost-sure decay of square-summable sequences.
\begin{lemm}\label{le:sequence}
Let $\{c_j\}_{j\in\mathbb{Z}_+}$ be a sequence of real (complex) numbers such that 
$\sum_{j=1}^{+\infty}|c_j|^2 < +\infty$.  
Then \(c_j = o(j^{-1/2})\) almost surely as$j\rightarrow+\infty$ in the sense of 
Definition~\ref{def:almost_sure_o}.

\end{lemm}

%
We now state the following observation.

\begin{lemm}\label{le:curl}
Let $\boldsymbol{g}(k; \boldsymbol{x}, \boldsymbol{y})$ be a function such that for each fixed $\boldsymbol{x}$, the map
\begin{align*}
\boldsymbol{y} \mapsto \boldsymbol{g}(k; \boldsymbol{x}, \boldsymbol{y})
\end{align*}
belongs to $\overset{\longrightarrow}{\operatorname{curl}}_{\partial D}( H^{\frac{3}{2}}_{\star}(\partial D))$.
Then $\boldsymbol{g}(k; \boldsymbol{x}, \boldsymbol{y})$ can be expanded as
\begin{align}\label{g:representation}
\boldsymbol g(k;\boldsymbol{x},\boldsymbol{y})
= \sum_{j=1}^{+\infty} \boldsymbol{g}_{j}(k;\boldsymbol{x})
\mathcal{N}^{-1}_{\partial D}[\boldsymbol{\phi}_j](\boldsymbol{y})
\end{align}
with respect to the inner product  
\(\langle\cdot,\cdot\rangle_{\mathcal{N}_{\partial D};
\overset{\longrightarrow}{\mathrm{curl}}_{\partial D}(H^{\frac32}_{\star}(\partial D))}\) 
defined in \eqref{E:Product}, with 
\begin{align*}
\boldsymbol{g}_{j}(k;\boldsymbol{x})
= \int_{\partial D} \boldsymbol g(k;\boldsymbol{x},\boldsymbol{y})\cdot
\boldsymbol{\phi}_{j}(\boldsymbol{y})\,\mathrm{d}s(\boldsymbol{y}).
\end{align*}

\end{lemm}
\begin{proof}
In fact, using Theorem \ref{th:decomposition1} yields 
\begin{align*}
\boldsymbol{g}_{j}(k;\boldsymbol x)&=\langle \boldsymbol g(k;\boldsymbol{x},\cdot),\mathcal{N}^{-1}_{\partial D}[\boldsymbol{\phi}_j]\rangle_{\mathcal{N}_{\partial D};\overset{\longrightarrow}{\mathrm{curl}}_{\partial D}(H^{\frac32}_{\star}(\partial D))}\\
&=\int_{\partial D}\mathcal{N}_{\partial D}[\boldsymbol g](k;\boldsymbol{x},\boldsymbol{y})\cdot\mathcal{N}^{-1}_{\partial D}[\boldsymbol{\phi}_{j}](\boldsymbol{y})\mathrm{d}s(\boldsymbol{y})\\
&=\int_{\partial D}\boldsymbol g(k;\boldsymbol{x},\boldsymbol{y})\cdot\boldsymbol{\phi}_{j}(\boldsymbol{y})\mathrm{d}s(\boldsymbol{y}),
\end{align*}
which completes the proof of the lemma.
\end{proof}

Let $(\cdot)^\top$ denote the transpose. Denote by $\boldsymbol{\mathfrak{E}}(k;\boldsymbol{x},\boldsymbol{y})$ the $3\times3$ matrix-valued function defined for $\boldsymbol{x}\neq\boldsymbol{y}$ as
\begin{align*}
\boldsymbol{\mathfrak{E}}(k;\boldsymbol{x},\boldsymbol{y}):= 
\begin{pmatrix}
\boldsymbol{G}_1(k;\boldsymbol{x},\boldsymbol{y}) \\
\boldsymbol{G}_2(k;\boldsymbol{x},\boldsymbol{y}) \\
\boldsymbol{G}_3(k;\boldsymbol{x},\boldsymbol{y})
\end{pmatrix}.
\end{align*}
Here, $\boldsymbol{G}_i(k;\boldsymbol{x},\boldsymbol{y})$ (for $i=1,2,3$) denotes the $i$-th row vector of the matrix, given  by
\begin{align*}
\boldsymbol{G}_1(k;\boldsymbol{x},\boldsymbol{y})&=
\begin{pmatrix} 
G(k;\boldsymbol{x},\boldsymbol{y}) & 0 & 0 
\end{pmatrix}, \\
\boldsymbol{G}_2(k;\boldsymbol{x},\boldsymbol{y})&=
\begin{pmatrix}
0 & G(k;\boldsymbol{x},\boldsymbol{y}) & 0 
\end{pmatrix}, \\
\boldsymbol{G}_3(k;\boldsymbol{x},\boldsymbol{y})&=
\begin{pmatrix}
0 & 0 & G(k;\boldsymbol{x},\boldsymbol{y})
\end{pmatrix},
\end{align*}
where $G(k;\boldsymbol{x},\boldsymbol{y})$ is defined by \eqref{Gk:sca-Fun}. By its construction, the matrix $\boldsymbol{\mathfrak{E}}(k;\boldsymbol{x}, \boldsymbol{y})$ is symmetric, i.e., $\boldsymbol{\mathfrak{E}}(k;\boldsymbol{x},\boldsymbol{y}) = \boldsymbol{\mathfrak{E}}(k;\boldsymbol{x},\boldsymbol{y})^{\top}$.

Define the exterior and interior \(\epsilon/2\)-neighborhoods of \(\partial D\):
\begin{align*}
\mathcal{T}^{+}_{\epsilon/2}(\partial D)
&:=\left\{\boldsymbol{x}\in\mathbb{R}^3\backslash \overline{D}:\mathrm{dist}(\boldsymbol{x},\partial D)<\epsilon/2\right\},\\
\mathcal{T}^{-}_{\epsilon/2}(\partial D)&:=\left\{\boldsymbol x\in D:\mathrm{dist}(\boldsymbol{x},\partial D)<\epsilon/2\right\}.
\end{align*}
For $\boldsymbol{x}\in \mathbb{R}^3\backslash{\mathcal{T}_\epsilon(\partial D)}$ and $\boldsymbol y\in\mathcal{T}^{\pm}_{\epsilon/2}(\partial D)$,  the following identities hold:
\begin{align}
\nabla_{\boldsymbol{x}}\times(G(k;\boldsymbol x,\boldsymbol y)\boldsymbol{\phi}_j(\boldsymbol y))&=\nabla_{\boldsymbol{x}}\times\boldsymbol{\mathfrak{E}}(k;\boldsymbol x,\boldsymbol y)\boldsymbol{\phi}_j(\boldsymbol y),\label{E:equ1}\\
\nabla_{\boldsymbol{x}}\times\nabla_{\boldsymbol{x}}\times(G(k;\boldsymbol x,\boldsymbol y)\boldsymbol{\phi}_j(\boldsymbol y))&=\nabla_{\boldsymbol{x}}\times\nabla_{\boldsymbol{x}}\times\boldsymbol{\mathfrak{E}}(k;\boldsymbol x,\boldsymbol y)\boldsymbol{\phi}_j(\boldsymbol y),\label{E:equ2}
\end{align}
where
\begin{align*}
\nabla_{\boldsymbol x}\times\boldsymbol{\mathfrak{E}}(k;\boldsymbol{x},\boldsymbol{y})&:=
\begin{pmatrix}
-(\nabla_{\boldsymbol x}\times(\boldsymbol{{G}}_1(k;\boldsymbol{ x},\boldsymbol{y})^{\top}))^{\top}\\
-(\nabla_{\boldsymbol x}\times(\boldsymbol{{G}}_2(k;\boldsymbol{x},\boldsymbol{y})^{\top}))^{\top}\\
-(\nabla_{\boldsymbol x}\times(\boldsymbol{{G}}_3(k;\boldsymbol{x},\boldsymbol{y})^{\top}))^{\top}
\end{pmatrix},\\
\nabla_{\boldsymbol x}\times\nabla_{\boldsymbol x}\times\boldsymbol{\mathfrak{E}}(k;\boldsymbol{x},\boldsymbol{y})&:=
\begin{pmatrix}
(\nabla_{\boldsymbol x}\times\nabla_{\boldsymbol x}\times(\boldsymbol{{G}}_1(k;\boldsymbol{ x},\boldsymbol{y})^{\top}))^{\top}\\
(\nabla_{\boldsymbol x}\times\nabla_{\boldsymbol x}\times(\boldsymbol{{G}}_2(k;\boldsymbol{x},\boldsymbol{y})^{\top}))^{\top}\\
(\nabla_{\boldsymbol x}\times\nabla_{\boldsymbol x}\times(\boldsymbol{{G}}_3(k;\boldsymbol{x},\boldsymbol{y})^{\top}))^{\top}
\end{pmatrix}.
\end{align*}
For  fixed  $\boldsymbol{x}\in \mathbb{R}^3\backslash\mathcal{T}_\epsilon(\partial D) $ and $\boldsymbol y\in\mathcal{T}^{\pm}_{\epsilon/2}(\partial D)$, if each row vector $\boldsymbol{A}_i(k;\boldsymbol{x},\boldsymbol{y})~(i=1,2,3)$ satisfies
\begin{align*}
\boldsymbol{A}_i(k;\boldsymbol{x},\cdot)^{\top}\in \mathbf{H}(\mathrm{curl},\mathcal{T}^{\pm}_{\epsilon/2}(\partial D)),
\end{align*}
then we say that the matrix-valued function
\begin{align*}
\boldsymbol{\mathfrak{A}}(k;\boldsymbol{x},\cdot)\in \mathbf{H}(\mathrm{curl}, \mathcal{T}^{\pm}_{\epsilon/2}(\partial D))^{3},
\end{align*}
where
\begin{align*}
\boldsymbol{\mathfrak{A}}(k;\boldsymbol{x},\boldsymbol{y}):=
\begin{pmatrix}
\boldsymbol{A}_1(k;\boldsymbol{x},\boldsymbol{y})\\
\boldsymbol{A}_2(k;\boldsymbol{x},\boldsymbol{y})\\
\boldsymbol{A}_3(k;\boldsymbol{x},\boldsymbol{y})
\end{pmatrix}.
\end{align*}
In particular, if $\boldsymbol{A}_i(k;\boldsymbol{x},\cdot)^{\top}\in \overset{\longrightarrow}{\mathrm{curl}}_{\partial D}(H^{\frac32}_{\star}(\partial D))$, then
\begin{align*} 
\boldsymbol{\mathfrak{A}}(k;\boldsymbol{x},\cdot)\in \overset{\longrightarrow}{\mathrm{curl}}_{\partial D}(H^{\frac32}_{\star}(\partial D))^{3},
\end{align*}
endowed with norm
\begin{align*}
\|\boldsymbol{\mathfrak{A}}(k;\boldsymbol{x},\cdot)\|^2_{\mathcal{N}_{\partial D};\overset{\longrightarrow}{\mathrm{curl}}_{\partial D}(H^{\frac32}_{\star}(\partial D))^{3}}=\int_{\partial D}\mathcal{N}_{\partial D}[\boldsymbol{\mathfrak{A}}](k;\boldsymbol{x},\cdot):\boldsymbol{\mathfrak{A}}(k;\boldsymbol{x},\cdot)\mathrm{d}s(\boldsymbol{y}),
\end{align*}
where $\boldsymbol{\mathfrak{A}}:\boldsymbol{\mathfrak{B}}:=\mathrm{tr}(\boldsymbol{\mathfrak{A}}\boldsymbol{\mathfrak{B}}^{\top})$ and
\begin{align*}
\mathcal{N}_{\partial D}[\boldsymbol{\mathfrak{A}}]:=
\begin{pmatrix}
\mathcal{N}_{\partial D}[\boldsymbol{A}^{\top}_1]^{\top}\\
\mathcal{N}_{\partial D}[\boldsymbol{A}^{\top}_2]^{\top}\\
\mathcal{N}_{\partial D}[\boldsymbol{A}^{\top}_3]^{\top}
\end{pmatrix}.
\end{align*}

It is obviously   that for $\boldsymbol{x}\in \mathbb{R}^3\backslash\mathcal{T}_\epsilon(\partial D) $ and $\boldsymbol y\in \mathcal{T}^{\pm}_{\epsilon/2}(\partial D)$,
\begin{align*}
\nabla_{\boldsymbol x}\times\boldsymbol{\mathfrak{E}}(k;\boldsymbol{x},\cdot)&\in\mathbf{H}(\mathrm{curl},
\mathcal{T}^{\pm}_{\epsilon/2}(\partial D))^{3},\\
\nabla_{\boldsymbol x}\times\nabla_{\boldsymbol x}\times\boldsymbol{\mathfrak{E}}(k;\boldsymbol{x},\cdot)&\in\mathbf{H}(\mathrm{curl},
\mathcal{T}^{\pm}_{\epsilon/2}(\partial D))^{3}.
\end{align*}
We now examine the structure of $\nabla_{\boldsymbol{x}}\times\boldsymbol{\mathfrak{E}}(k;\boldsymbol{x},\boldsymbol{y})$. 
 When $\boldsymbol{y}$ is restricted to $\partial D$, for each $i=1,2,3$ we have
\begin{align*}
\boldsymbol{\nu}_{\boldsymbol{y}}\times(\nabla_{\boldsymbol x}\times(\boldsymbol{G}_i(k;\boldsymbol{x},\cdot)^{\top})\times \boldsymbol{\nu}_{\boldsymbol{y}})\in \mathbf{H}^{-\frac{1}{2}}_t(\mathrm{curl},\partial D),
\end{align*}
where $\nabla_{\boldsymbol x}\times(\boldsymbol{G}_i(k;\boldsymbol{x},\cdot)^{\top})\times \boldsymbol{\nu}_{\boldsymbol{y}}\in \mathbf{H}^{-\frac{1}{2}}_t(\mathrm{div},\partial D)$. This implies that
\begin{align*}
\boldsymbol{\nu}_{\boldsymbol{y}}\times(\nabla_{\boldsymbol x}\times(\boldsymbol{G}_i(k;\boldsymbol{x},\cdot)^{\top})\times \boldsymbol{\nu}_{\boldsymbol{y}})\stackrel{\eqref{DE:curl}}{=}\nabla_{\partial D,\boldsymbol{y}}\Gamma_i(k;\boldsymbol{x},\boldsymbol{y})+ \overset{\longrightarrow}{\mathrm{curl}}_{\partial D,\boldsymbol{y}}\Lambda_i(k;\boldsymbol{x},\boldsymbol{y}),
\end{align*}
where
\begin{align*}
\Gamma_i(k;\boldsymbol{x},\cdot)\in H^{\frac{1}{2}}(\partial D),\quad \Lambda_i(k;\boldsymbol{x},\cdot)\in H^{\frac32}_{\star}(\partial D).
\end{align*}
Hence, we obtain the decomposition
\begin{align}\label{c:decomposition}
\nabla_{\boldsymbol x}\times\boldsymbol{\mathfrak{E}}(k;\boldsymbol{x},\boldsymbol{y})&=\nabla_{\boldsymbol x}\times\boldsymbol{\mathfrak{E}}(k;\boldsymbol{x},\boldsymbol{y})_{\mathrm{normal}}+\nabla_{\boldsymbol x}\times\boldsymbol{\mathfrak{E}}(k;\boldsymbol{x},\boldsymbol{y})_{\nabla_{\partial D}}\nonumber\\
&\quad+\nabla_{\boldsymbol x}\times\boldsymbol{\mathfrak{E}}(k;\boldsymbol{x},\boldsymbol{y})_{\overset{\longrightarrow}{\mathrm{curl}}_{\partial D}},
\end{align}
where
\begin{align*}
\nabla_{\boldsymbol x}\times\boldsymbol{\mathfrak{E}}(k;\boldsymbol{x},\boldsymbol{y})_{\mathrm{normal}}&=
\begin{pmatrix}
-((\nabla_{\boldsymbol x}\times(\boldsymbol{G}_1(k;\boldsymbol{ x},\boldsymbol{y})^{\top})\cdot \boldsymbol{\nu}_{\boldsymbol{y}})\boldsymbol\nu_{\boldsymbol y})^{\top}\\
-((\nabla_{\boldsymbol x}\times(\boldsymbol{G}_2(k;\boldsymbol{ x},\boldsymbol{y})^{\top})\cdot \boldsymbol{\nu}_{\boldsymbol{y}})\boldsymbol\nu_{\boldsymbol y})^{\top}\\
-((\nabla_{\boldsymbol x}\times(\boldsymbol{G}_3(k;\boldsymbol{ x},\boldsymbol{y})^{\top})\cdot \boldsymbol{\nu}_{\boldsymbol{y}})\boldsymbol\nu_{\boldsymbol y})^{\top}
\end{pmatrix},\\
\nabla_{\boldsymbol x}\times\boldsymbol{\mathfrak{E}}(k;\boldsymbol{x},\boldsymbol{y})_{\nabla_{\partial D}}&=
\begin{pmatrix}
-(\nabla_{\partial D,\boldsymbol{y}}\Gamma_1(k;\boldsymbol{x},\boldsymbol{y}))^{\top}\\
-(\nabla_{\partial D,\boldsymbol{y}}\Gamma_2(k;\boldsymbol{x},\boldsymbol{y}))^{\top}\\
-(\nabla_{\partial D,\boldsymbol{y}}\Gamma_3(k;\boldsymbol{x},\boldsymbol{y}))^{\top}
\end{pmatrix},\\
\nabla_{\boldsymbol x}\times\boldsymbol{\mathfrak{E}}(k;\boldsymbol{x},\boldsymbol{y})_{\overset{\longrightarrow}{\mathrm{curl}}_{\partial D}}&=
\begin{pmatrix}
-(\overset{\longrightarrow}{\mathrm{curl}}_{\partial D,\boldsymbol{y}}\Lambda_1(k;\boldsymbol{x},\boldsymbol{y}))^{\top}\\
-(\overset{\longrightarrow}{\mathrm{curl}}_{\partial D,\boldsymbol{y}}\Lambda_2(k;\boldsymbol{x},\boldsymbol{y}))^{\top}\\
-(\overset{\longrightarrow}{\mathrm{curl}}_{\partial D,\boldsymbol{y}}\Lambda_3(k;\boldsymbol{x},\boldsymbol{y}))^{\top}
\end{pmatrix}.
\end{align*}
A completely analogous argument as the proof of formula \eqref{c:decomposition} yields a similar decomposition for the double curl:
\begin{align}\label{cc:decomposition}
\nabla_{\boldsymbol x}\times\nabla_{\boldsymbol x}\times\boldsymbol{\mathfrak{E}}(k;\boldsymbol{x},\boldsymbol{y})&=\nabla_{\boldsymbol x}\times\nabla_{\boldsymbol x}\times\boldsymbol{\mathfrak{E}}(k;\boldsymbol{x},\boldsymbol{y})_{\mathrm{normal}}\nonumber\\
&\quad+\nabla_{\boldsymbol x}\times\nabla_{\boldsymbol x}\times\boldsymbol{\mathfrak{E}}(k;\boldsymbol{x},\boldsymbol{y})_{\nabla_{\partial D}}
\nonumber\\
&\quad+\nabla_{\boldsymbol x}\times\nabla_{\boldsymbol x}\times\boldsymbol{\mathfrak{E}}(k;\boldsymbol{x},\boldsymbol{y})_{\overset{\longrightarrow}{\mathrm{curl}}_{\partial D}}.
\end{align}

It remains to complete the proof of Theorem \ref{th:location1}.

\begin{proof}[Proof of Theorem \ref{th:location1}]

Beginning with formulas \eqref{E:mode} and \eqref{H:mode}, we first investigate the almost sure decay properties of $\nabla\times\vec{\mathcal{S}}^{k}_{\partial D}[\boldsymbol{\phi}_j]$. 

For $\boldsymbol{x}\in \mathbb{R}^3\backslash\mathcal{T}_\epsilon(\partial D) $ and $\boldsymbol y\in{\partial D}$, using \eqref{E:equ1} and \eqref{c:decomposition} yields that
\begin{align*}
\nabla\times\vec{\mathcal{S}}^{k}_{\partial D}[\boldsymbol{\phi}_j](\boldsymbol{x})&=\int_{\partial D}\nabla_{\boldsymbol{x}}\times(G(k;\boldsymbol x,\boldsymbol y)\boldsymbol{\phi}_j(\boldsymbol y))\mathrm{d}s(\boldsymbol{y})\\
&=\int_{\partial D}\nabla_{\boldsymbol x}\times\boldsymbol{\mathfrak{E}}(k;\boldsymbol{x},\boldsymbol{y})_{\overset{\longrightarrow}{\mathrm{curl}}_{\partial D}}\boldsymbol{\phi}_j(\boldsymbol y)\mathrm{d}s(\boldsymbol{y})\\
&=
\begin{pmatrix}
-\int_{\partial D}\overset{\longrightarrow}{\mathrm{curl}}_{\partial D,\boldsymbol{y}}\Lambda_1(k;\boldsymbol{x},\boldsymbol{y})\cdot \boldsymbol{\phi}_{j}(\boldsymbol{y})\mathrm{d}s(\boldsymbol{y})\\
-\int_{\partial D}\overset{\longrightarrow}{\mathrm{curl}}_{\partial D,\boldsymbol{y}}\Lambda_2(k;\boldsymbol{x},\boldsymbol{y})\cdot \boldsymbol{\phi}_{j}(\boldsymbol{y})\mathrm{d}s(\boldsymbol{y})\\
-\int_{\partial D}\overset{\longrightarrow}{\mathrm{curl}}_{\partial D,\boldsymbol{y}}\Lambda_3(k;\boldsymbol{x},\boldsymbol{y})\cdot \boldsymbol{\phi}_{j}(\boldsymbol{y})\mathrm{d}s(\boldsymbol{y})
\end{pmatrix}.
\end{align*}
Moreover, it follows from Lemma \ref{le:curl} that for $i=1,2,3,$
\begin{align*}
\overset{\longrightarrow}{\mathrm{curl}}_{\partial D,\boldsymbol{y}}\Lambda_i(k;\boldsymbol{x},\boldsymbol{y})\stackrel{\eqref{g:representation}}{=}\sum_{j=1}^{+\infty}\int_{\partial D}\overset{\longrightarrow}{\mathrm{curl}}_{\partial D,\boldsymbol{y}}\Lambda_i(k;\boldsymbol{x},\boldsymbol{y})\cdot\boldsymbol{\phi}_{j}(\boldsymbol{y})\mathrm{d}s(\boldsymbol{y})\mathcal{N}^{-1}_{\partial D}[\boldsymbol{\phi}_j](\boldsymbol{y}).
\end{align*}
As a result, we can derive
\begin{align*}
\sum^{+\infty}_{j=1}\left|\int_{\partial D}\nabla_{\boldsymbol{x}}\times(G(k;\boldsymbol x,\boldsymbol y)\boldsymbol{\phi}_j(\boldsymbol y))\mathrm{d}s(\boldsymbol{y})\right|^2&=\sum^{+\infty}_{j=1}\sum^3_{i=1}\left|\int_{\partial D}\overset{\longrightarrow}{\mathrm{curl}}_{\partial D,\boldsymbol{y}}\Lambda_i(k;\boldsymbol{x},\boldsymbol{y})\cdot \boldsymbol{\phi}_{j}(\boldsymbol{y})\mathrm{d}s(\boldsymbol{y})\right|^2\\
&=\left\|\nabla_{\boldsymbol x}\times\boldsymbol{\mathfrak{E}}(k;\boldsymbol{x},\cdot)_{\overset{\longrightarrow}{\mathrm{curl}}_{\partial D}}\right\|^2_{\mathcal{N}_{\partial D};\overset{\longrightarrow}{\mathrm{curl}}_{\partial D}(H^{\frac32}_{\star}(\partial D))^{3}}.
\end{align*}
Similarly, it follows from \eqref{E:equ2} and \eqref{cc:decomposition} that
\begin{align*}
&\sum^{+\infty}_{j=1}\left|\int_{\partial D}\nabla_{\boldsymbol{x}}\times\nabla_{\boldsymbol{x}}\times(G(k;\boldsymbol x,\boldsymbol y)\boldsymbol{\phi}_j(\boldsymbol y))\mathrm{d}s(\boldsymbol{y})\right|^2\\
&=\left\|\nabla_{\boldsymbol{x}}\times\nabla_{\boldsymbol x}\times\boldsymbol{\mathfrak{E}}(k;\boldsymbol{x},\cdot)_{\overset{\longrightarrow}{\mathrm{curl}}_{\partial D}}\right\|^2_{\mathcal{N}_{\partial D};\overset{\longrightarrow}{\mathrm{curl}}_{\partial D}(H^{\frac32}_{\star}(\partial D))^{3}}.
\end{align*}
Hence, there is come constant $C>0$ such that for any compact set $K$ in $\mathbb{R}^3\backslash\mathcal{T}_\epsilon(\partial D) $,
\begin{align*}
\sum^{+\infty}_{j=1}\int_{K}|\boldsymbol{E}_{j}(\boldsymbol x)|^2\mathrm{d}\boldsymbol{x}
&\leq|\mu_{D}|\sum^{+\infty}_{j=1}\int_{K}\left|\int_{\partial D}\nabla_{\boldsymbol{x}}\times(G(k_D;\boldsymbol x,\boldsymbol y)\boldsymbol{\phi}_j(\boldsymbol y))\mathrm{d}s(\boldsymbol{y})\right|^2\mathrm{d}\boldsymbol{x}\\
&\quad+
\sum^{+\infty}_{j=1}\int_{K}\left|\int_{\partial D}\nabla_{\boldsymbol{x}}\times\nabla_{\boldsymbol{x}}\times(G(k_D;\boldsymbol x,\boldsymbol y)\boldsymbol{\phi}_j(\boldsymbol y))\mathrm{d}s(\boldsymbol{y})\right|^2\mathrm{d}\boldsymbol{x}\\
&\leq|\mu_{D}|\int_{K}\left\|\nabla_{\boldsymbol x}\times\boldsymbol{\mathfrak{E}}(k_D;\boldsymbol{x},\cdot)_{\overset{\longrightarrow}{\mathrm{curl}}_{\partial D}}\right\|^2_{\mathcal{N}_{\partial D};\overset{\longrightarrow}{\mathrm{curl}}_{\partial D}(H^{\frac32}_{\star}(\partial D))^{3}}\mathrm{d}\boldsymbol{x}\\
&\quad+\int_{K}\left\|\nabla_{\boldsymbol x}\times\nabla_{\boldsymbol x}\times\boldsymbol{\mathfrak{E}}(k_D;\boldsymbol{x},\cdot)_{\overset{\longrightarrow}{\mathrm{curl}}_{\partial D}}\right\|^2_{\mathcal{N}_{\partial D};\overset{\longrightarrow}{\mathrm{curl}}_{\partial D}(H^{\frac32}_{\star}(\partial D))^{3}}\mathrm{d}\boldsymbol{x}\\
&\leq C,
\end{align*}
with $\mu_D=\mu_e=1,k_D=k_e=\omega$ or $\mu_D=\mu_c=-\tau_j,k_D=k_c=\omega\tau_j$.

The same argument applied to the magnetic field \(\boldsymbol{H}_j\) yields that for any compact set $K$ in $\mathbb{R}^3\backslash\mathcal{T}_\epsilon(\partial D) $,
\begin{align*}
\sum^{+\infty}_{j=1}\int_{K}|\boldsymbol{H}_{j}(\boldsymbol x)|^2\mathrm{d}\boldsymbol{x}\leq C.
\end{align*}
Thanks to Lemma \ref{le:sequence}, we have
\begin{align*}
\int_{K}|\boldsymbol{E}_{j}(\boldsymbol x)|^2\mathrm{d}\boldsymbol{x}=o(j^{-1}),\quad\int_{K}|\boldsymbol{H}_{j}(\boldsymbol x)|^2\mathrm{d}\boldsymbol{x}=o(j^{-1})
\end{align*}
almost surely as $j \rightarrow +\infty$ in the sense of Definition~\ref{def:almost_sure_o}. This completes the proof of Theorem~\ref{th:location1}.
\end{proof}

\section{Boundary localization on spherical nanoparticles}\label{sec:6}

This section is devoted to the scattering problem for spherical particles 
$D$ under plane wave incidence. For the special case of balls, explicit expressions for the weak plasmon sequence are derived, and its boundary localization behavior is established. The proof of Theorem \ref{th:location2} is then completed.

Let us first introduce some definitions. Let $D=\delta B$, where $B$ is the ball centered at the origin with radius $\tilde{r}$. Denote by $\mathbb{S}$ the unit sphere given by \eqref{eq:sphere}.  For $n=1,2,\cdots$ and $m=-n,\cdots,n$, denote by $Y^m_n$ the spherical harmonics defined on $\mathbb{S}$. Let  $h^{(1)}_n$ and $j_n$ be  the spherical Hankel function of the first kind of order $n$ and the spherical Bessel function of order $n$, respectively. For any vector $\boldsymbol x\neq 0$, write $\hat{\boldsymbol x}=\boldsymbol{x}/| \boldsymbol{x}|\in \mathbb{S}$, and let $\nabla_{\mathbb S}$ represent the surface gradient on $\mathbb S$.
Define the vector spherical harmonics by
\begin{align}\label{phi1-2}
\boldsymbol{\phi}^m_{1,n}(\hat{\boldsymbol x})=\frac{1}{\sqrt{n(n+1)}}\nabla_{\mathbb S}Y^m_n(\hat{\boldsymbol x}),\quad \boldsymbol{\phi}^m_{2,n}(\hat{\boldsymbol x})=\hat{\boldsymbol x}\times \boldsymbol{\phi}^m_{1,n}(\hat{\boldsymbol x})
\end{align}
for $n=1,2,\cdots$ and $m=-n,\cdots,n$. The vector spherical harmonics $\{\boldsymbol{\phi}^m_{1,n}\} $ and $\{\boldsymbol{\phi}^m_{2,n}\} $ form a complete orthogonal basis of the space  $\mathbf{H}^{-\frac{1}{2}}_t(\mathbb S)$ of tangential vector fields on $\mathbb{S}$ (cf. \cite{Ammari2016Mathematical}).

Additionally, we assume that the following fact.
\begin{assu}\label{assumption3}
For  $l=1,2, n=1,2,\cdots$, and $m=-n,\cdots,n$, we assume
\begin{align*}
\epsilon_e&=\mu_e=1,\quad \epsilon_c=\mu_c=-\tau_{l,n},
\end{align*}
where $\tau_{l,n}>0$ and $\tau_{l,n}\neq1$.
\end{assu}
In fact, Assumption~\ref{assumption3} is introduced merely for notational convenience and is essentially equivalent to Assumption~\ref{assumption2}.

\begin{prop}\label{th:mode2}
Under Assumption~\ref{assumption3}, if the Maxwell system \eqref{eq:maxwell} exhibits weak SPRs 
for sufficiently small $\delta>0$, then for every $l=1,2,n=1,2,\cdots,$ and $m=-n,\cdots,n$, the following holds:

\begin{enumerate}
\item[\rm{(1)}] System \eqref{OP2:B} admits a family of  the nontrivial solutions
\[
          \begin{pmatrix}
          \boldsymbol{\phi}^{m}_{l,n} \\[1mm]
          \omega\boldsymbol{\phi}^{m}_{l,n}
          \end{pmatrix}
          \in \bigl(\mathbf{H}^{-\frac12}_t(\mathbb{S})\bigr)^2,
          \]
associated with the material parameters $\tau_{l,n}$ introduced in Assumption~\ref{assumption3}.

\item[\rm{(2)}] One has the following weak plasmon sequence $\{(\boldsymbol{E}_{l,n},\boldsymbol{H}_{l,n})\}_{l=1,2;n\in\mathbb{Z}_+}$, where
\begin{align}\label{eq:E}
\boldsymbol{E}_{l,n}(\boldsymbol{x})=
\left\{ \begin{aligned}
&\mu_e\nabla\times\vec{\mathcal{S}}^{k_e}_{\partial D}[\boldsymbol{\phi}^m_{l,n}](\boldsymbol x)+\nabla\times\nabla\times\vec{\mathcal{S}}^{k_e}_{\partial D}[\boldsymbol{\phi}^m_{l,n}](\boldsymbol x), && \boldsymbol{x}\in\mathbb{R}^3\backslash\overline{D},\\
&\mu_c\nabla\times\vec{\mathcal{S}}^{k_c}_{\partial D}[\boldsymbol{\phi}^m_{l,n}](\boldsymbol x)+\nabla\times\nabla\times\vec{\mathcal{S}}^{k_c}_{\partial D}[\boldsymbol{\phi}^{m}_{l,n}](\boldsymbol x), &&\boldsymbol{x}\in D,
\end{aligned}\right.
\end{align}
and
\begin{align}\label{eq:H}
\boldsymbol H_{l,n}(x)=
\left\{ \begin{aligned}
&-\frac{\mathrm{i}}{\omega}\nabla\times\nabla\times\vec{\mathcal{S}}^{k_e}_{\partial D}[\boldsymbol{\phi}^m_{l,n}](\boldsymbol x)-\frac{\mathrm{i}}{\omega\mu_e}k^2_e\nabla\times\vec{\mathcal{S}}^{k_e}_{\partial D}[\boldsymbol{\phi}^m_{l,n}](\boldsymbol x), &&\boldsymbol{x}\in\mathbb{R}^3\backslash\overline{D},\\
&-\frac{\mathrm{i}}{\omega}\nabla\times\nabla\times\vec{\mathcal{S}}^{k_c}_{\partial D}[\boldsymbol{\phi}^m_{l,n}](\boldsymbol x)-\frac{\mathrm{i}}{\omega\mu_c}k^2_c\nabla\times\vec{\mathcal{S}}^{k_c}_{\partial D}[\boldsymbol{\phi}^m_{l,n}](\boldsymbol x), &&\boldsymbol{x}\in D.
\end{aligned}\right.
\end{align}
\end{enumerate}
\end{prop}

Before presenting the proof of this proposition, we first focus on the spectral analysis of the static MNP operator defined on spherical surfaces $\partial B$.
\begin{lemm}
On the space $\mathbf{H}^{-\frac{1}{2}}_{t}(\partial B)$, the following equalities hold:
\begin{align}
\mathcal{M}_{\partial B}[\boldsymbol{\phi}^m_{1,n}]&=\lambda_{1,n}\boldsymbol{\phi}^m_{1,n},\qquad n=1,2,\cdots; m=-n,\cdots,n,\label{OPS:NN}\\
\mathcal{M}_{\partial B}[\boldsymbol{\phi}^m_{2,n}]&=\lambda_{2,n}\boldsymbol{\phi}^m_{2,n},\qquad n=1,2,\cdots; m=-n,\cdots,n,\label{OPS:N}
\end{align}
where 
\begin{align*}
\lambda_{1,n}=-\frac{1}{2(2n+1)},\qquad \lambda_{2,n}=\frac{1}{2(2n+1)}.
\end{align*}
\end{lemm}
\begin{proof}
Without loss of generality, we set $\tilde{r}=1$. Combining (3.26) and (3.27) from \cite{deng2019on}, (2.4.116) from \cite{Nedelec2002acoustic}, and the jump relation \eqref{SS:jump}, we derive \eqref{OPS:NN} and \eqref{OPS:N}.
\end{proof}

After deriving the spectral system for the operator  $\mathcal{M}_{\partial B}$, we now characterize the spectrum of the operator  
$\tilde{\mathscr{A}}_{\partial B,j}(0)$ defined in \eqref{E:operator-A}.

\begin{lemm}\label{le:eigen-A}
The eigenvalues and their corresponding eigenfunctions for the operator $\tilde{\mathscr{A}}_{\partial B,j}(0)$ are given as follows:
\begin{align*}
\tilde{\mathscr{A}}_{\partial B,j}(0)
\begin{pmatrix}
\boldsymbol{\phi}^m_{l,n}\\
\boldsymbol{\phi}^m_{l,n}
\end{pmatrix}
=\left(\frac{1-\tau_{l,n}}{2(1+\tau_{l,n})}-\lambda_{l,n}\right)
\begin{pmatrix}
\boldsymbol{\phi}^m_{l,n}\\
\boldsymbol{\phi}^m_{l,n}
\end{pmatrix}
\end{align*}
for $l=1,2,n=1,2,\cdots$, and $m=-n,\cdots,n$.
\end{lemm}

We are now in a position to prove Proposition \ref{th:mode2}.

\begin{proof}[Proof of Proposition \ref{th:mode2}]
According to Lemma \ref{le:eigen-A}, we take, for $l=1,2,n=1,2,\cdots$, and $m=-n,\cdots,n$,
\begin{align*}
\begin{pmatrix}
\tilde{\boldsymbol{\psi}}\\[1mm]
\omega\tilde{\boldsymbol{\varphi}}
\end{pmatrix}
=
\begin{pmatrix}
\boldsymbol{\phi}^{m}_{l,n}\\[1mm]
\omega\boldsymbol{\phi}^{m}_{l,n}
\end{pmatrix}.
\end{align*}
If the resonance condition $\frac{1-\tau_{l,n}}{2(1+\tau_{l,n})}=\lambda_{l,n}$ holds, system \eqref{EQB:equation} undergoes weak SPRs as described in \eqref{OP2:B}. This establishes part $\mathrm{(1)}$ of Proposition \ref{th:mode2}.

Additionally,  we can obtain  the weak plasmon sequence $\{(\boldsymbol{E}_{l,n},\boldsymbol{H}_{l,n})\}_{l=1,2;n\in\mathbb{Z}_+}$ in terms of \eqref{EX:E} and \eqref{EX:H}, namely, part $\mathrm{(2)}$ of Proposition \ref{th:mode2} holds.

The proof is complete.
\end{proof}

Before proceeding to  the proof of Theorem \ref{th:location2}, we need some auxiliary results. For a wave number $k>0$, the function
\begin{align}\label{E:v}
v_{n,m}(k;\boldsymbol x)=h^{(1)}_n(k|\boldsymbol x|)Y^m_n(\hat{\boldsymbol x})
\end{align}
satisfies the Helmholtz equation in $\mathbb{R}^3\setminus\{0\}$ together with the Sommerfeld radiation condition,
that is,
\begin{align*}
\left\{ \begin{aligned}
&\Delta v+k^2 v=0 &&\text{in }\mathbb{R}^3\setminus \{0\},\\[1mm]
&\lim_{|\boldsymbol x|\rightarrow+\infty}|\boldsymbol x|\left(\frac{\partial v_{n,m}}{\partial |\boldsymbol x|}(k;\boldsymbol x)-\mathrm{i}kv_{n,m}(k;\boldsymbol x)\right)=0.
\end{aligned}\right.
\end{align*}
Similarly, define $\tilde{v}_{n,m}$ by
\begin{align}\label{E:tildev}
\tilde{v}_{n,m}(k;\boldsymbol x)=j_{n}(k|\boldsymbol x|)Y^m_n(\hat{\boldsymbol x}),
\end{align}
which satisfies the Helmholtz equation in $\mathbb{R}^3$.

The solutions to Maxwell's equations admit a multipole expansion in vector spherical harmonics (cf. \cite{Ammari2016Mathematical}). We define the exterior transverse electric (TE) and transverse magnetic (TM) multipoles (which satisfy the Sommerfeld radiation condition), along with their interior counterparts, as follows:
\begin{align*}
\begin{pmatrix}
\boldsymbol{E}^{TE}_{n,m}\\[1mm]
\boldsymbol{H}^{TE}_{n,m}
\end{pmatrix}
,\quad
\begin{pmatrix}
\boldsymbol{E}^{TM}_{n,m}\\[1mm]
\boldsymbol{H}^{TM}_{n,m}
\end{pmatrix}
,\quad
\begin{pmatrix}
\tilde{\boldsymbol E}^{TE}_{n,m}\\[1mm]
\tilde{\boldsymbol H}^{TE}_{n,m}
\end{pmatrix}
,\quad 
\begin{pmatrix}
\tilde{\boldsymbol E}^{TM}_{n,m}\\[1mm]
\tilde{\boldsymbol H}^{TM}_{n,m}
\end{pmatrix}
.
\end{align*}
That is,
\begin{align*}
\begin{pmatrix}
\boldsymbol{E}^{TE}_{n,m}(k;\boldsymbol x)\\[1mm]
\boldsymbol{H}^{TE}_{n,m}(k;\boldsymbol x)
\end{pmatrix}
&=
\begin{pmatrix}
-\sqrt{n(n+1)}h^{(1)}_n(k|\boldsymbol x|)\boldsymbol{\phi}^m_{2,n}(\hat{\boldsymbol x})\\
-\dfrac{\mathrm{i}}{\omega\mu}\nabla\times\left(-\sqrt{n(n+1)}h^{(1)}_n(k|\boldsymbol x|)\boldsymbol{\phi}^m_{2,n}(\hat{\boldsymbol x})\right)
\end{pmatrix}
,\\
\begin{pmatrix}
\boldsymbol{E}^{TM}_{n,m}(k;\boldsymbol x)\\[1mm]
{\boldsymbol H}^{TM}_{n,m}(k;\boldsymbol x)
\end{pmatrix}
&=
\begin{pmatrix}
\dfrac{\mathrm{i}}{\omega\epsilon}\nabla\times\left(-\sqrt{n(n+1)}h^{(1)}_n(k|\boldsymbol x|)\boldsymbol{\phi}^m_{2,n}(\hat{\boldsymbol x})\right)\\[1mm]
-\sqrt{n(n+1)}h^{(1)}_n(k|\boldsymbol x|)\boldsymbol{\phi}^m_{2,n}(\hat{\boldsymbol x})
\end{pmatrix}
,\\
\begin{pmatrix}
\tilde{\boldsymbol{E}}^{TE}_{n,m}(k;\boldsymbol x)\\
\tilde{\boldsymbol{H}}^{TE}_{n,m}(k;\boldsymbol x)
\end{pmatrix}
&=
\begin{pmatrix}
-\sqrt{n(n+1)}j_n(k|\boldsymbol x|)\boldsymbol{\phi}^m_{2,n}(\hat{\boldsymbol x})\\
-\dfrac{\mathrm{i}}{\omega\mu}\nabla\times\left(-\sqrt{n(n+1)}j_n(k|\boldsymbol x|)\boldsymbol{\phi}^m_{2,n}(\hat{\boldsymbol x})\right)
\end{pmatrix}
,\\
\begin{pmatrix}
\tilde{\boldsymbol{E}}^{TM}_{n,m}(k;\boldsymbol x)\\
\tilde{\boldsymbol H}^{TM}_{n,m}(k;\boldsymbol x)
\end{pmatrix}
&=
\begin{pmatrix}
\dfrac{\mathrm{i}}{\omega\epsilon}\nabla\times\left(-\sqrt{n(n+1)}j_n(k|\boldsymbol x|)\boldsymbol{\phi}^m_{2,n}(\hat{\boldsymbol x})\right)\\
-\sqrt{n(n+1)}j_n(k|\boldsymbol x|)\boldsymbol{\phi}^m_{2,n}(\hat{\boldsymbol x})
\end{pmatrix}
.
\end{align*}
Defining the composite spherical Hankel and Bessel functions as
\begin{align*}
\mathcal{H}_n(k|\boldsymbol x|)&=h^{(1)}_{n}(k|\boldsymbol x|)+k|\boldsymbol x|(h^{(1)}_n)'(k|\boldsymbol x|),\\[1mm] 
\mathcal{J}_n(k|\boldsymbol x|)&=j_{n}(k|\boldsymbol x|)+k|\boldsymbol x|j'_n(k|\boldsymbol x|),
\end{align*}
the curl of the (exterior and interior) TE multipoles is then given by (cf. \cite[(A.5)~and~(A.6)]{Ammari2016Mathematical})
\begin{align*}
\nabla\times \boldsymbol{E}^{TE}_{n,m}(k;\boldsymbol x)&=\frac{\sqrt{n(n+1)}}{|\boldsymbol x|}\mathcal{H}_n(k|\boldsymbol x|)\boldsymbol{\phi}^m_{1,n}(\hat{\boldsymbol x})+\frac{n(n+1)}{|\boldsymbol x|}h^{(1)}_n(k|\boldsymbol x|)Y^m_n(\hat{\boldsymbol x})\hat{\boldsymbol x},\\
\nabla\times \tilde{\boldsymbol{E}}^{TE}_{n,m}(k;\boldsymbol x)&=\frac{\sqrt{n(n+1)}}{|\boldsymbol x|}\mathcal{J}_n(k|\boldsymbol x|)\boldsymbol{\phi}^m_{1,n}(\hat{\boldsymbol x})+\frac{n(n+1)}{|\boldsymbol x|}j_n(k|\boldsymbol x|)Y^m_n(\hat{\boldsymbol x})\hat{\boldsymbol x}.
\end{align*}

Additionally, the functions $j_n,y_n$, and $h^{(1)}_n$ admit the series representations (cf. \cite{Colton1998Inverse}):
\begin{align*}
j_n(z)&=\sum_{p=0}^{+\infty}\frac{(-1)^pz^{n+2p}}{2^pp!1\cdot3\cdot~\cdots~\cdot(2n+2p+1)},\\[1mm]
y_n(z)&=-\frac{(2n)!}{2^n n!}\sum_{p=0}^{+\infty}\frac{(-1)^pz^{2p-n-1}}{2^pp!(-2n+1)\cdot(-2n+3)\cdot~\cdots~\cdot(-2n+2p-1)},\\[1mm]
h^{(1)}_{n}(z)&=j_n(z)+\mathrm{i}y_n(z).
\end{align*}
It is clear that, as $n\rightarrow+\infty$,
\begin{align*}
j_n(z)&=\frac{z^n}{(2n+1)!!}\left(1+\mathcal{O}(\frac1n)\right),\\ h^{(1)}_{n}(z)&=\frac{(2n-1)!!}{\mathrm{i} z^{n+1}}\left(1+\mathcal{O}(\frac1n)\right),\\
\mathcal{J}_n(z)&=\frac{(n+1)z^n}{(2n+1)!!}\left(1+\mathcal{O}(\frac1n)\right),\\ \mathcal{H}_{n}(z)&=-\frac{n(2n-1)!!}{\mathrm{i} z^{n+1}}\left(1+\mathcal{O}(\frac1n) \right)
\end{align*}
uniformly on compact subsets of $(0,+\infty)$.

Based on the preliminary results, for $l = 1, 2$, we now present a lemma on the asymptotic behavior of the curl of $\vec{\mathcal{S}}^{k_D}_{\partial D}[\boldsymbol{\phi}^m_{l,n}]$ in the exterior $(r' > r)$ and interior $(r' < r)$ domains of the spherical boundary $\partial D$ as $n \rightarrow+\infty$.

\begin{lemm}\label{le:decay}
For a spherical boundary 
 $\partial D=\{|\boldsymbol y|=r\}$ with radius
$r=\delta\tilde{r}$
  and vector spherical harmonics 
$\boldsymbol{\phi}^m_{l,n}~(l=1,2)$  defined in \eqref{phi1-2}, we have the following asymptotic expansions:

\begin{enumerate}
\item[\rm{(1)}] For $r'=|\boldsymbol x|>r$ (exterior domain), one has
\begin{align}
\nabla\times\vec{\mathcal{S}}^{k_e}_{\partial D}[\boldsymbol{\phi}^m_{1,n}](\boldsymbol x)&=\left(\frac{r}{r'}\right)^{n+1}\left(\frac{1}{2}+\mathcal{O}(\frac{1}{n})\right)
\boldsymbol{\phi}^m_{2,n}(\hat{\boldsymbol{x}}),\nonumber\\
\nabla\times\vec{\mathcal{S}}^{k_e}_{\partial D}[\boldsymbol{\phi}^m_{2,n}](\boldsymbol x)
&=
\left(\frac{r}{r'}\right)^{n+2}\left(\frac{1}{2}+\mathcal{O}(\frac{1}{n})\right)\boldsymbol{\phi}^m_{1,n}(\hat{\boldsymbol{x}})\nonumber\\
&\quad-\left(\frac{r}{r'}\right)^{n+2}\left(\frac{1}{2}+\mathcal{O}(\frac{1}{n})\right)Y_{n,m}
(\hat{\boldsymbol{x}})\hat{\boldsymbol{x}},\nonumber\\
\nabla\times\nabla\times\vec{\mathcal{S}}^{k_e}_{\partial D}[\boldsymbol{\phi}^m_{1,n}](\boldsymbol x)&=
\frac{n}{r'}\left(\frac{r}{r'}\right)^{n+1}\left(\frac{1}{2}+\mathcal{O}(\frac{1}{n})\right)
\boldsymbol{\phi}^m_{1,n}(\hat{\boldsymbol{x}})\nonumber\\
&\quad-\frac{\sqrt{n(n+1)}}{r'}\left(\frac{r}{r'}\right)^{n+1}
\left(\frac{1}{2}+\mathcal{O}(\frac{1}{n})\right)Y_{n,m}(\hat{\boldsymbol{x}})
\hat{\boldsymbol{x}},\nonumber\\
\nabla\times\nabla\times\vec{\mathcal{S}}^{k_e}_{\partial D}[\boldsymbol{\phi}^m_{2,n}](\boldsymbol x)&=\frac{k_e(k_er)}{2n+1}\left(\frac{r}{r'}\right)^{n+1}\left(1+\mathcal{O}(\frac{1}{n})\right)\boldsymbol{\phi}^m_{2,n}(\hat{\boldsymbol{x}})\nonumber
\end{align}
as $n\rightarrow+\infty$.

\item[\rm{(2)}] For $r'=|\boldsymbol x|<r$ (interior domain), one has
\begin{align}
\nabla\times\vec{\mathcal{S}}^{k_c}_{\partial D}[\boldsymbol{\phi}^m_{1,n}](\boldsymbol x)&=\left(\frac{r'}{r}\right)^{n}\left(\frac{1}{2}+\mathcal{O}(\frac{1}{n})\right)
\boldsymbol{\phi}^m_{2,n}(\hat{\boldsymbol{x}}),\nonumber\\
\nabla\times\vec{\mathcal{S}}^{k_c}_{\partial D}[\boldsymbol{\phi}^m_{2,n}](\boldsymbol x)&=
-\left(\frac{r'}{r}\right)^{n-1}\left(\frac{1}{2}+\mathcal{O}(\frac{1}{n})\right)\boldsymbol{\phi}^m_{1,n}(\hat{\boldsymbol{x}})\nonumber\\
&\quad-\left(\frac{r'}{r}\right)^{n-1}\left(\frac{1}{2}+\mathcal{O}(\frac{1}{n})\right)Y_{n,m}
(\hat{\boldsymbol{x}})\hat{\boldsymbol{x}},\nonumber\\
\nabla\times\nabla\times\vec{\mathcal{S}}^{k_c}_{\partial D}[\boldsymbol{\phi}^m_{1,n}](\boldsymbol x)&=
\frac{n}{r'}\left(\frac{r'}{r}\right)^{n}\left(\frac{1}{2}+\mathcal{O}(\frac{1}{n})\right)
\boldsymbol{\phi}^m_{1,n}(\hat{\boldsymbol{x}})\nonumber\\
&\quad+\frac{\sqrt{n(n+1)}}{r'}\left(\frac{r'}{r}\right)^{n}
\left(\frac{1}{2}+\mathcal{O}(\frac{1}{n})\right)Y_{n,m}(\hat{\boldsymbol{x}})
\hat{\boldsymbol{x}},\nonumber\\
\nabla\times\nabla\times\vec{\mathcal{S}}^{k_c}_{\partial D}[\boldsymbol{\phi}^m_{2,n}](\boldsymbol{x})&=\frac{k_c(k_cr)}{2n+1}\left(\frac{r'}{r}\right)^{n}\left(1+\mathcal{O}(\frac{1}{n})\right)\boldsymbol{\phi}^m_{2,n}(\hat{\boldsymbol{x}})\nonumber
\end{align}
as $n\rightarrow+\infty$.
\end{enumerate}
\end{lemm}

\begin{proof}

Let $\mathbf{q}$ be an arbitrary fixed vector in $\mathbb{R}^3$. Then for all $\boldsymbol{x}, \boldsymbol{y} \in \mathbb{R}^3$ with $|\boldsymbol{x}| > |\boldsymbol{y}|$, the dyadic Green's function admits the expansion (cf.\cite{Ammari2016Mathematical}):
\begin{align*}
G(k;\boldsymbol{x},\boldsymbol{y})\mathbf{q}&=-\sum^{+\infty}_{n=1}\frac{\mathrm{i}k}{n(n+1)}\frac{\epsilon}{\mu}\sum^{n}_{m=-n}\boldsymbol{E}^{TM}_{n,m}(k;\boldsymbol x)\overline{\tilde{\boldsymbol{E}}^{TM}_{n,m}(k;\boldsymbol y)}\cdot \mathbf{q}\nonumber\\
&\quad+\sum^{+\infty}_{n=1}\frac{\mathrm{i}k}{n(n+1)}\sum^{n}_{m=-n}\boldsymbol{E}^{TE}_{n,m}(k;\boldsymbol x)\overline{\tilde{\boldsymbol{E}}^{TE}_{n,m}(k;\boldsymbol y)}\cdot \mathbf{q}\nonumber\\
&\quad-\frac{\mathrm{i}}{k}\sum^{+\infty}_{n=1}\sum^{n}_{m=-n}\nabla v_{n,m}(k;\boldsymbol x)\overline{\nabla \tilde{v}_{n,m}(k;\boldsymbol y)}\cdot \mathbf{q},
\end{align*}
while  for $|\boldsymbol x|<|\boldsymbol y|$,
\begin{align*}
G(k;\boldsymbol{x},\boldsymbol{y})\mathbf{q}&=-\sum^{+\infty}_{n=1}\frac{\mathrm{i}k}{n(n+1)}\frac{\epsilon}{\mu}\sum^{n}_{m=-n}
\overline{\tilde{\boldsymbol{E}}^{TM}_{n,m}(k;\boldsymbol x)}\boldsymbol{E}^{TM}_{n,m}(k;\boldsymbol y)\cdot \mathbf{q}\nonumber\\
&\quad+\sum^{+\infty}_{n=1}\frac{\mathrm{i}k}{n(n+1)}\sum^{n}_{m=-n}\overline{\tilde{\boldsymbol{E}}^{TE}_{n,m}(k;\boldsymbol x)}\boldsymbol{E}^{TE}_{n,m}(k;\boldsymbol y)\cdot \mathbf{q}\nonumber\\
&\quad-\frac{\mathrm{i}}{k}\sum^{+\infty}_{n=1}\sum^{n}_{m=-n}\overline{\nabla \tilde{v}_{n,m}(k;\boldsymbol x)}\nabla v_{n,m}(k;\boldsymbol y)\cdot \mathbf{q}.
\end{align*}
Here, $v_{n,m}$ and $\tilde{v}_{n,m}$ are defined by formulas \eqref{E:v} and \eqref{E:tildev}, respectively. Furthermore, from the Maxwell system \eqref{eq:maxwell}, we deduce that for every $\boldsymbol{x} \in \mathbb{R}^3\setminus\partial D$,
\begin{align}
\left\{ \begin{aligned}\label{E:qqE}
&\nabla\times\nabla\times \boldsymbol{E}-k^2_D\boldsymbol{E}=0,\\[2mm]
&\nabla\times\nabla\times \boldsymbol{H}-k^2_D\boldsymbol{H}=0.
\end{aligned}\right.
\end{align}
Combining  \eqref{E:tildev}, the first equation of \eqref{E:qqE},  with  the explicit forms of  $\tilde{\boldsymbol{E}}^{TM}_{n,m}$ and $\nabla\times \tilde{\boldsymbol{E}}^{TE}_{n,m}$, we obtain for $|\boldsymbol x|>|\boldsymbol y|$,
\begin{align*}
&\nabla\times\vec{\mathcal{S}}^{k_e}_{\partial D}[\boldsymbol{\phi}^m_{1,n}](\boldsymbol x)\\&=\nabla\times\int_{\partial D}
\Big(-\sum^{+\infty}_{p=1}\frac{\mathrm{i}k_e}{p(p+1)}\frac{\epsilon}{\mu}\sum^{p}_{q=-p}\boldsymbol{E}^{TM}_{p,q}(k_e;\boldsymbol x)\overline{\tilde{\boldsymbol{E}}^{TM}_{p,q}(k_e;\boldsymbol y)}\cdot \boldsymbol{\phi}^m_{1,n}(\hat{\boldsymbol{y}})\nonumber\\
&\quad+\sum^{+\infty}_{p=1}\frac{\mathrm{i}k_e}{p(p+1)}\sum^{p}_{q=-p}\boldsymbol{E}^{TE}_{p,q}(k_e;\boldsymbol x)\overline{\tilde{\boldsymbol{E}}^{TE}_{p,q}(k_e;\boldsymbol y)}\cdot  \boldsymbol{\phi}^m_{1,n}(\hat{\boldsymbol{y}})\nonumber\\
&\quad-\frac{\mathrm{i}}{k_e}\sum^{+\infty}_{p=1}\sum^{p}_{q=-p}\nabla v_{p,q}(k_e;\boldsymbol x)\overline{\nabla \tilde{v}_{p,q}(k_e;\boldsymbol y)}\cdot  \boldsymbol{\phi}^m_{1,n}(\hat{\boldsymbol{y}})\Big)\mathrm{d}s(\boldsymbol y)\nonumber\\
&=-\frac{k_e}{r\sqrt{n(n+1)}}\frac{1}{\omega\mu_e}\nabla\times\boldsymbol{E}^{TM}_{n,m}(k_e;\boldsymbol x)\mathcal{J}_n(k_er)\int_{\partial D}\boldsymbol{\phi}^m_{1,n}(\hat{\boldsymbol{y}})\cdot \boldsymbol{\phi}^m_{1,n}(\hat{\boldsymbol{y}})\mathrm{d}s(\boldsymbol y)\nonumber\\
&=\mathrm{i}k_erh^{(1)}_n(k_er')\mathcal{J}_n(k_er)\boldsymbol{\phi}^m_{2,n}(\hat{\boldsymbol{x}}).
\end{align*}
A similar computation yields, again for $|\boldsymbol x|>|\boldsymbol y|$,
\begin{align*}
&\nabla\times\vec{\mathcal{S}}^{k_e}_{\partial D}[\boldsymbol{\phi}^m_{2,n}](\boldsymbol x)\\
&=-\frac{\mathrm{i}k_e}{\sqrt{n(n+1)}}\nabla\times\boldsymbol{E}^{TE}_{n,m}(k_e;\boldsymbol x)j_n(k_er)\int_{\partial D}\boldsymbol{\phi}^m_{2,n}(\hat{\boldsymbol{y}})\cdot \boldsymbol{\phi}^m_{2,n}(\hat{\boldsymbol{y}})\mathrm{d}s(\boldsymbol y)\nonumber\\
&=-\frac{\mathrm{i}k_er^2}{r'}\mathcal{H}_n(k_er')j_n(k_er)\boldsymbol{\phi}^m_{1,n}(\hat{\boldsymbol{x}})
-\frac{\mathrm{i}k_er^2\sqrt{n(n+1)}}{r'}h^{(1)}_n(k_er')j_n(k_er)Y_{n,m}(\hat{\boldsymbol{x}})\hat{\boldsymbol{x}},\\
&\nabla\times\nabla\times\vec{\mathcal{S}}^{k_e}_{\partial D}[\boldsymbol{\phi}^m_{1,n}](\boldsymbol x)\\
&=-\frac{k_e}{r\sqrt{n(n+1)}}\frac{1}{\omega\mu_e}\nabla\times\nabla\times\boldsymbol{E}^{TM}_{n,m}(k_e;\boldsymbol x)\mathcal{J}_n(k_er)\int_{\partial D}\boldsymbol{\phi}^m_{1,n}(\hat{\boldsymbol{y}})\cdot \boldsymbol{\phi}^m_{1,n}(\hat{\boldsymbol{y}})\mathrm{d}s(\boldsymbol y)\nonumber\\
&=-\frac{\mathrm{i}k_er}{r'}\mathcal{H}_n(k_er')\mathcal{J}_n(k_er)\boldsymbol{\phi}^m_{1,n}(\hat{\boldsymbol{x}})
-\frac{\mathrm{i}k_er\sqrt{n(n+1)}}{r'}h^{(1)}_n(k_er')\mathcal{J}_n(k_er)Y_{n,m}(\hat{\boldsymbol{x}})\hat{\boldsymbol{x}},\\
&\nabla\times\nabla\times\vec{\mathcal{S}}^{k_e}_{\partial D}[\boldsymbol{\phi}^m_{2,n}](\boldsymbol x)\\
&=-\frac{\mathrm{i}k_e}{\sqrt{n(n+1)}}\nabla\times\nabla\times\boldsymbol{E}^{TE}_{n,m}(k_e;\boldsymbol x)j_n(k_er)\int_{\partial D}\boldsymbol{\phi}^m_{2,n}(\hat{\boldsymbol{y}})\cdot \boldsymbol{\phi}^m_{2,n}(\hat{\boldsymbol{y}})\mathrm{d}s(\boldsymbol y)\nonumber\\
&=\mathrm{i}k_e^3r^2h^{(1)}_n(k_er')j_n(k_er)\boldsymbol{\phi}^m_{2,n}(\hat{\boldsymbol{x}}).
\end{align*}

In addition, one has that for $|\boldsymbol x|<|\boldsymbol y|$,
\begin{align*}
&\nabla\times\vec{\mathcal{S}}^{k_c}_{\partial D}[\boldsymbol{\phi}^m_{1,n}](\boldsymbol x)\\
&=\frac{k_c}{r\sqrt{n(n+1)}}\frac{1}{\omega\mu_c}\overline{\nabla\times\tilde{\boldsymbol{E}}^{TM}_{n,m}(k_c;\boldsymbol x)}\mathcal{H}_n(k_cr)\int_{\partial D}\boldsymbol{\phi}^m_{1,n}(\hat{\boldsymbol{y}})\cdot \boldsymbol{\phi}^m_{1,n}(\hat{\boldsymbol{y}})\mathrm{d}s(\boldsymbol y)\nonumber\\
&=\mathrm{i}k_crj_n(k_cr')\mathcal{H}_n(k_cr)\boldsymbol{\phi}^m_{2,n}(\hat{\boldsymbol{x}}),\\
&\nabla\times\vec{\mathcal{S}}^{k_c}_{\partial D}[\boldsymbol{\phi}^m_{2,n}](\boldsymbol x)\\
&=\frac{\mathrm{i}k_c}{\sqrt{n(n+1)}}\overline{\nabla\times\tilde{\boldsymbol{E}}^{TE}_{n,m}(k_c;\boldsymbol x)}h^{(1)}_n(k_cr)\int_{\partial D}\boldsymbol{\phi}^m_{2,n}(\hat{\boldsymbol{y}})\cdot \boldsymbol{\phi}^m_{2,n}(\hat{\boldsymbol{y}})\mathrm{d}s(\boldsymbol y)\nonumber\\
&=\frac{\mathrm{i}k_cr^2}{r'}\mathcal{J}_n(k_cr')h^{(1)}_n(k_cr)\boldsymbol{\phi}^m_{1,n}(\hat{\boldsymbol{x}})
+\frac{\mathrm{i}k_cr^2\sqrt{n(n+1)}}{r'}j_n(k_cr')h^{(1)}_n(k_cr)Y_{n,m}(\hat{\boldsymbol{x}})\hat{\boldsymbol{x}},\\
&\nabla\times\nabla\times\vec{\mathcal{S}}^{k_c}_{\partial D}[\boldsymbol{\phi}^m_{1,n}](\boldsymbol x)\\
&=\frac{k_c}{r\sqrt{n(n+1)}}\frac{1}{\omega\mu_c}\overline{\nabla\times\nabla\times\tilde{\boldsymbol{E}}^{TM}_{n,m}(k_c;\boldsymbol x)}\mathcal{H}_n(k_cr)\int_{\partial D}\boldsymbol{\phi}^m_{1,n}(\hat{\boldsymbol{y}})\cdot \boldsymbol{\phi}^m_{1,n}(\hat{\boldsymbol{y}})\mathrm{d}s(\boldsymbol y)\nonumber\\
&=\frac{-\mathrm{i}k_cr}{r'}\mathcal{J}_n(k_cr')\mathcal{H}_n(k_cr)\boldsymbol{\phi}^m_{1,n}(\hat{\boldsymbol{x}})
+\frac{-\mathrm{i}k_cr\sqrt{n(n+1)}}{r'}j_n(k_cr')\mathcal{H}_n(k_cr)Y_{n,m}(\hat{\boldsymbol{x}})\hat{\boldsymbol{x}},\\
&\nabla\times\nabla\times\vec{\mathcal{S}}^{k_c}_{\partial D}[\boldsymbol{\phi}^m_{2,n}](\boldsymbol x)\\
&=\frac{\mathrm{i}k_c}{\sqrt{n(n+1)}}\overline{\nabla\times\nabla\times\tilde{\boldsymbol{E}}^{TE}_{n,m}(k_c;\boldsymbol x)}h^{(1)}_n(k_cr)\int_{\partial D}\boldsymbol{\phi}^m_{2,n}(\hat{\boldsymbol{y}})\cdot \boldsymbol{\phi}^m_{2,n}(\hat{\boldsymbol{y}})\mathrm{d}s(\boldsymbol y)\nonumber\\
&=-\mathrm{i}k_c^3r^2j_n(k_cr')h^{(1)}_n(k_cr)\boldsymbol{\phi}^m_{2,n}(\hat{\boldsymbol{x}}).
\end{align*}
Substituting  the explicit expressions for  $j_n,h^{(1)}_n,\mathcal{J}_n$, and $\mathcal{H}_n$, we can get the conclusion of this lemma, which completes the proof.
\end{proof}

With the help of Lemma \ref{le:decay}, the  weak plasmon sequence of system \eqref{eq:maxwell} decay exponentially fast as $n\rightarrow+\infty$ whenever $\boldsymbol{x}$ lies away from the boundary \(\partial D\).
\begin{proof}[Proof of Theorem \ref{th:location2}]
Clearly, Theorem \ref{th:location2} follows directly from the representations \eqref{E:mode} and \eqref{H:mode} together with the exponential decay estimate provided by Lemma~\ref{le:decay}.

\end{proof}

\textbf {Declarations}

{\textbf Conflict of interest} \, The author has no Conflict of interest.

\providecommand{\noopsort}[1]{}\providecommand{\singleletter}[1]{#1}%


\end{document}